\documentclass[a4paper,11pt]{article}

\usepackage[notext]{kpfonts}
\usepackage{baskervald}

\usepackage{amsfonts,amssymb,amsmath,amsthm,abstract,color}
\usepackage[mathscr]{euscript}
\usepackage[ps,all,arc,rotate]{xy}
\usepackage[lmargin=1in,rmargin=1in,tmargin=1in,bmargin=1in]{geometry}
\usepackage{fancyhdr}
\usepackage{dsfont}
\usepackage{pb-diagram}
\usepackage{enumitem}
\usepackage[backref=page]{hyperref}
\usepackage[usenames,dvipsnames]{xcolor}
\usepackage{tikz-cd}
\usepackage{extpfeil}
\usepackage{float}
\hypersetup{colorlinks=true,citecolor=NavyBlue,linkcolor=Brown,urlcolor=Orange}

\usepackage{stmaryrd}
\usepackage[capitalise]{cleveref}
\numberwithin{equation}{section}


\usepackage{titlesec}

\titlespacing*{\section}
{0pt}{3ex plus 1ex minus .2ex}{2.5ex plus .2ex}
\titlespacing*{\subsection}
{0pt}{2ex plus 1ex minus .2ex}{1.5ex plus .2ex}
\titlespacing*{\subsubsection}
{0pt}{2ex plus 1ex minus .2ex}{1.5ex}

\newcommand{\sectionpoint}{\if \@empty\titlesec \else}              

\newcommand*{\justifyheading}{\raggedright}
\titleformat{\section}{\normalfont \Large \bfseries \justifyheading}{\thesection.}{0.5em}{}
\titleformat{\subsection}[runin]{\normalfont \large \bfseries \justifyheading}{\thesubsection.}{0.5em}{}[.]
\titleformat{\subsubsection}[runin]{\normalfont \normalsize \bfseries \justifyheading}{\thesubsubsection.}{0.5em}{}[  \protect{\rule[3pt]{10pt}{0.5pt}}]

\theoremstyle{plain}
\newtheorem*{theorem*}{Theorem}
\newtheorem{theorem}{Theorem}[section]
\newtheorem{prop}[theorem]{Proposition}

\newtheorem{lemma}[theorem]{Lemma}
\newtheorem{cor}[theorem]{Corollary}

\theoremstyle{definition}
\newtheorem{definition}[theorem]{Definition}

\theoremstyle{remark}
\newtheorem{rmk}[theorem]{Remark}

\newtheorem{question}[theorem]{Question}

\renewcommand{\tilde}{\widetilde}     

\newcommand{\CC}{\mathbb{C}}

\newcommand{\PP}{\mathbb{P}}
\newcommand{\QQ}{\mathbb{Q}}

\newcommand{\ZZ}{\mathbb{Z}}

\newcommand{\h}{\mathfrak{h}}

\DeclareMathOperator{\rk}{rk}

\newcommand{\HM}{\mathsf{HM}}
\newcommand{\VHS}{\mathsf{VHS}}

\renewcommand{\D}{\mathscr{D}}
\newcommand{\Def}{\mathrm{Def}}
\newcommand{\DR}{\mathrm{DR}}

\newcommand{\GL}{\mathrm{GL}}
\newcommand{\gr}{\mathrm{gr}}
\newcommand{\Hilb}{\mathrm{Hilb}}
\newcommand{\IC}{\mathrm{IC}}
\newcommand{\id}{\mathsf{id}}
\newcommand{\ICH}{\mathcal{IC}}

\newcommand{\Perv}{\mathsf{Perv}}

\newcommand{\R}{\mathbf{R}}
\newcommand{\p}{\mathsf{p}}
\newcommand{\Sym}{\mathrm{Sym}}

\begin{document}

\title{Twisted Hodge groups and deformation theory of Hilbert schemes of points on surfaces via Hodge modules}

\author{Lie Fu}
\date{ }
\maketitle

\begin{abstract}
Given a smooth compact complex surface together with a holomorphic line bundle on it, using the theory of Hodge modules, we compute the twisted Hodge groups/numbers of Hilbert schemes (or Douady spaces) of points on the surface with values in the naturally associated line bundle.
This proves an amended version of Boissi\`ere's conjecture proposed by the author in his joint work with Belmans and Krug, and extends G\"ottsche--Soergel's formula for Hodge numbers and G\"ottsche's formula for refined $\chi_y$-genera to any compact complex surface, without K\"ahlerness assumption. As an application, we determine the tangent space and the obstruction space of the formal deformation theory of Hilbert schemes of points on compact complex surfaces. Analogous results are obtained for nested Hilbert schemes.
\end{abstract}

\tableofcontents

%\thanks{\textit{2020 Mathematics Subject Classification:}  14C05, 14D15, 14C30, 14D07.}
%\thanks{\textit{Key words and phrases:} Hilbert schemes, Douady spaces, Hodge numbers, deformation theory, Hodge modules, perverse sheaves, decomposition theorem. }
%\thanks{The author is supported by the University of Strasbourg Institute for Advanced Study (USIAS), and by the Agence Nationale de la Recherche (ANR) under projects ANR-20-CE40-0023 and ANR-24-CE40-4098. }

\newpage

%%%%%%%%%%%%%%%%%%%%%
% Content begins here
%%%%%%%%%%%%%%%%%%%%%

\section{Introduction}

\subsection{Hilbert schemes of points on surfaces}

Let $S$ be a smooth projective complex surface. For a positive integer $n$, the $n$-th symmetric power $S^{(n)}$, defined as the quotient of $S^n$ by the natural action of the symmetric group $\mathfrak{S}_n$, is a $2n$-dimensional projective variety with Gorenstein quotient singularities (\cite{Aramova}). 
A natural smooth birational model of $S^{(n)}$ is provided by the so-called Hilbert scheme of length-$n$ subschemes on $S$, denoted by  $\Hilb^nS$, which is a $2n$-dimensional smooth projective variety by a theorem of Fogarty \cite{Fogarty}. The Hilbert--Chow morphism 
\begin{equation}
\pi\colon \Hilb^n(S)\to S^{(n)}
\end{equation}
is a crepant resolution of singularities, that is, $\pi^*\omega_{S^{(n)}}\cong \omega_{\Hilb^n(S)}$.
The above definitions make sense more generally for a compact complex surface $S$, with $\Hilb^n(S)$ replaced by the $n$-th Douady space and the Hilbert--Chow morphism replaced by the Douady--Barlet morphism. Although almost all the results in this paper, whenever make sense, work in the more general complex analytic setting, we will keep using the algebraic notation and terminology.  For standard references, see the lecture notes of G\"ottsche \cite{Gottsche-LectureNotes}, Nakajima \cite{Nakajima-LectureNotesHilb}, and Lehn \cite{Lehn-LectureHilbert}.

A general principle is that invariants of $\Hilb^nS$ can often be expressed in terms of the same type of invariants of $S$. To put our results in the context, let us list below some known results exemplifying this principle. Considering all the Hilbert schemes $\{\Hilb^nS\}_{n\in \mathbb{N}}$ simultaneously often leads to a neater formula, and renders extra representation-theoretic structures more transparent. For this reason, results in this paper are stated in an all-$n$-together form, as isomorphisms of multi-graded vector spaces or equalities of generating series. From them, the formulas for an individual Hilbert scheme, which we omit, can be easily deduced.

\begin{itemize}
	\item \textit{Betti numbers} (G\"ottsche  \cite{Gottsche-BettiNumber}):
	\begin{equation}
	\label{eqn:BettiNumbers}
	\sum_{n\geq 0}\sum_{i\geq 0} \operatorname{b}_i(\Hilb^nS)x^it^n=\prod_{k\geq 1}\prod_{i\geq 0} (1-(-1)^ix^{i+2k-2}t^k)^{-(-1)^i\operatorname{b}_i(S)}.
	\end{equation}

	\item \textit{Cohomology} and \textit{Hodge structures} (G\"ottsche--Soergel \cite{GottscheSoergel}):
	\begin{equation}
	\label{eqn:Hilb-HodgeStructure}
	\bigoplus_{n\geq 0} H^*(\Hilb^nS, \QQ)(n)[2n]t^n\cong \Sym^\bullet\left(\bigoplus_{k\geq 1} H^*(S, \QQ)(1)[2] t^k\right),
	\end{equation}
	where $-(m):=-\otimes\QQ(m)$ stands for the $m$-th Tate twist of a Hodge structure, $[m]$ is the standard degree shifting, and $\Sym^\bullet$ is the total symmetric power of the bigraded \textit{super} vector space $\bigoplus_{k\geq 1} H^*(S, \QQ)(1)[2] t^k$  subject to the super-sign rule with respect to the cohomological degree $*$.
	Moreover,  thanks to the work of Nakajima \cite{Nakajima-Annals, Nakajima-LectureNotesHilb} and Grojnowski \cite{Grojnowski}, the left-hand side of \eqref{eqn:Hilb-HodgeStructure} is identified with the Fock space representation of the Heisenberg Lie algebra associated with $H^*(S, \QQ)$.
	
	The next two formulas are direct consequences of \eqref{eqn:Hilb-HodgeStructure} (see \cite[Theorem 2.3.14]{Gottsche-LectureNotes}):
	\item  \textit{Hodge numbers}:
	\begin{equation}
	\label{eqn:HodgeNumbers}
	\sum_{n\geq 0}\sum_{p, q\geq 0} h^{p,q}(\Hilb^nS)x^py^qt^n=\prod_{k\geq 1}\prod_{p,q \geq 0} (1-(-1)^{p+q}x^{p+k-1}y^{q+k-1}t^k)^{-(-1)^{p+q}h^{p,q}(S)}.
	\end{equation}
	\item \textit{$\chi_y$-genera}:
	\begin{equation}
	\label{eqn:ChiyGenera}
	\sum_{n\geq 0} \chi_{-y}(\Hilb^nS)t^n=\exp\left(\sum_{m\geq 1} \frac{t^m}{m}\cdot \frac{\chi_{-y^m}(S)}{1-(yt)^m}\right).
	\end{equation}
	
	\item \textit{Hochschild homology} (by the Hochschild--Kostant--Rosenberg isomorphism and \eqref{eqn:Hilb-HodgeStructure}):
	\begin{equation}
	\label{eqn:HH}
	\bigoplus_{n\geq 0} HH_*(\Hilb^nS)t^n\cong \Sym^\bullet\left(\bigoplus_{k\geq 1} HH_*(S)t^k\right).
	\end{equation}
	As a generalization:
	\item \textit{Hochschild--Serre cohomology} (Belmans--Fu--Krug \cite{BFK}), for any $k\in \ZZ$:
	\begin{equation}
	\label{eqn:HS}
	\bigoplus_{n\geq 0} HS^*_k(\Hilb^nS)t^n\cong \Sym^\bullet\left(\bigoplus_{i\geq 1} HS^*_{1+(k-1)i}(S)t^i\right).
	\end{equation}
	Setting $k=0$:
	\item \textit{Hochschild cohomology} (Belmans--Fu--Krug \cite{BFK}) :
	\begin{equation}
	\label{eqn:HH*}
	\bigoplus_{n\geq 0} HH^*(\Hilb^nS)t^n\cong \Sym^\bullet\left(\bigoplus_{i\geq 1} HS^*_{1-i}(S)t^i\right).
	\end{equation}
	\item \textit{In the Grothendieck ring of varieties} (G\"ottsche \cite{Gottsche-MotiveHilb}):
	\begin{equation}
	\sum_{n\geq 0} [\Hilb^nS]\cdot \mathbb{L}^{-n}t^n = \Sym^\bullet\left(\sum_{k\geq 1} [S]\cdot \mathbb{L}^{-1} t^k\right).
	\end{equation}
	\item \textit{Chow motives} (de Cataldo--Migliorini \cite{DeCataldoMigliorini-MotiveOfHilbertScheme}):
	\begin{equation}
	\bigoplus_{n\geq 0} \h(\Hilb^nS)(n)t^n\cong \Sym^\bullet\left(\bigoplus_{k\geq 1} \h(S)(1) t^k\right).
	\end{equation}
	\item \textit{Derived categories} of perfect complexes (Bridgeland--King--Reid \cite{BridgelandKingReid} and Haiman \cite{Haiman}):
	\begin{equation}
	\operatorname{D^b_{coh}}(\Hilb^nS)\cong \operatorname{D^b_{coh}}([S^n/\mathfrak{S}_n])\cong \Sym^n\operatorname{D^b_{coh}}(S),
	\end{equation}
	where $\Sym^n$ is in the sense of Ganter--Kapranov \cite{GanterKapranov}.
\end{itemize}

Now if the smooth projective complex surface $S$ is equipped with a line bundle $L$, the line bundle $L^{\boxtimes n}:=p_1^*L\otimes \cdots \otimes p_n^*L$ on $S^n $ is endowed with a natural $\mathfrak{S}_n$-linearization. We define the following line bundle on the symmetric power $S^{(n)}$ by the  invariant push-forward:
\begin{equation}
\label{eqn:L(n)}
L_{(n)}:=\varpi_*(L^{\boxtimes n})^{\mathfrak{S}_n},
\end{equation}
where $\varpi\colon S^n \to S^{(n)}$ is the quotient map and $-^{\mathfrak{S}_n}$ stands for the (exact) functor of taking invariants.  Pulling back via the Hilbert--Chow morphism  $\pi\colon \Hilb^n(S)\to S^{(n)}$,  we define the following\footnote{Our notation for this natural line bundle $L_n$ follows \cite{EllingsrudGottscheLehn}. This bundle is denoted by $\mu(L)$ in \cite{Gottsche-RefinedVerlinde-2020}, while the notation $L_n$ in \cite{Gottsche-RefinedVerlinde-2020} refers to what we denote by $L_{(n)}$.} naturally associated line bundle $L_n$ on $\Hilb^nS$:
\begin{equation}
\label{eqn:Ln}
L_n:=\pi^*L_{(n)}.
\end{equation}
Note that $(\mathscr{O}_S)_n\cong\mathscr{O}_{\Hilb^nS}$ and $(\omega^{\otimes k}_S)_n\cong \omega^{\otimes k}_{\Hilb^nS}$ for any $k\in \ZZ$.

Similarly to the above exemplified principle, many invariants of the pair $(\Hilb^nS, L_n)$ can be expressed in terms of invariants of the pair $(S, L)$. For example:

\begin{itemize}
	\item \textit{$\chi_y$-genera with coefficients} (G\"ottsche \cite[Corollary 1.2]{Gottsche-RefinedVerlinde-2020} as reformulated in  \cite[Proposition 5.11]{BFK}):
	\begin{equation}
	\label{eqn:ChiyGenusCoeff}
	\sum_{n\geq 0} \chi_{-y}(\Hilb^nS, L_n)t^n=\prod_{k\geq 1}\prod_{p\geq 0} \left(1-y^{p+k-1}t^k\right)^{-(-1)^p\chi(S,\Omega_S^p\otimes L^{\otimes k})}.
	\end{equation}
	The right-hand side can also be reorganized into an expression in terms of the $\chi_y$-genera of $S$ with coefficients in powers of $L$, see \cite[Remark 5.12]{BFK} or \eqref{eqn:ChiyGenusCoeff2} below.
	\item \textit{Hochschild homology with coefficients} (Belmans--Fu--Krug \cite[Corollary 3.22]{BFK}):
	\begin{equation}\label{eqn:HHCoeff}
	\bigoplus_{n\geq 0} HH_*(\Hilb^nS, L_n)t^n\cong \Sym^{\bullet} \left(\bigoplus_{k\geq 1} HH_*(S, L^{\otimes k})t^k\right).
	\end{equation}
\end{itemize}

The main goal of the paper, achieved in \Cref{thm:main-cohomology} and \Cref{thm:main-numbers} below, is to establish a common refinement of G\"ottsche--Soergel's  \eqref{eqn:HodgeNumbers}, G\"ottsche's \eqref{eqn:ChiyGenusCoeff} and Belmans--Fu--Krug's \eqref{eqn:HHCoeff}, by computing the so-called \textit{twisted Hodge groups} of the pair $(\Hilb^nS, L_n)$.

\subsection{Main result: twisted Hodge groups/numbers}

Given a compact complex manifold $X$ equipped with a holomorphic line bundle $\mathcal{L}$, for any integers $p,q$, we define, following Boissi\`ere \cite{MR2932167}, the $(p,q)$-th \textit{twisted Hodge group} as
\begin{equation}
\label{eqn:TwistedHodge}
H^{p,q}(X, \mathcal{L}):=H^q(X, \Omega_X^p\otimes \mathcal{L}),
\end{equation}
whose dimension
\begin{equation}
h^{p,q}(X, \mathcal{L}):=\dim H^q(X, \Omega_X^p\otimes \mathcal{L})
\end{equation}
is called the $(p,q)$-th \textit{twisted Hodge number} of the pair $(X, \mathcal{L})$. Note that $H^{p,q}(X, \mathcal{L})$ is the Dolbeault cohomology of $X$ with values in $\mathcal{L}$.  The usual Hodge numbers correspond to the case where $\mathcal{L}\cong \mathscr{O}_X$. If $\mathcal{L}$ is not trivial or $X$ is not K\"ahler, the twisted Hodge numbers are in general \textit{not} the $(p,q)$-summands of some natural Hodge structure; this makes their computation often inaccessible by topological methods.

Keeping the notation as before, the first main result of the paper is to express all the twisted Hodge groups  $H^{p,q}(\Hilb^n(S), L_n)$ in terms of the twisted Hodge groups of $S$ with values in tensor powers of $L$:

\begin{theorem}
	\label{thm:main-cohomology}
	Let $S$ be a smooth compact complex surface. Let $L$ be a holomorphic line bundle on $S$. Let $L_n$ be the line bundle in \eqref{eqn:Ln} on the Douady space $\Hilb^nS$. We have a canonical isomorphism of tri-graded vector spaces:
	\begin{equation}
	\label{eqn:main-cohomology}
	\bigoplus_{n\geq 0}\bigoplus_{p,q=0}^{2n} H^{p,q}(\Hilb^nS, L_n)x^py^qt^n\cong \Sym^\bullet \left(\bigoplus_{k\geq 1}\bigoplus_{p,q=0}^{2}H^{p,q}(S, L^{\otimes k})x^{p+k-1}y^{q+k-1}t^k\right).
	\end{equation}
	where  $\Sym^\bullet$ is taken in the super sense \footnote{That is,  $\Sym^\bullet$ is the total symmetric power of the tri-graded super vector space  $\bigoplus_{k\geq 1}\bigoplus_{p,q=0}^{2}H^{p,q}(S, L^{\otimes k})x^{p+k-1}y^{q+k-1}t^k$ where the parity of a homogeneous element is the parity of its total degree of  $x$ and $y$.} with respect to the total degree of  $x$ and $y$ but in the ordinary sense with respect to the grading given by the degree of $t$.

	Writing in a more succinct way the isomorphism of tri-graded vector spaces:
	\begin{equation}
	\bigoplus_{n\geq 0} H^{\#,\star}(\Hilb^nS, L_n)t^n\cong \Sym^\bullet \left(\bigoplus_{k\geq 1}H^{\#,\star}(S, L^{\otimes k})[1-k, 1-k]t^k\right).
	\end{equation}
	where $[1-k, 1-k]$ denotes the shift of the bigrading $(\#, \star)$.
\end{theorem}

Taking the dimensions in \eqref{eqn:main-cohomology}, we get the generating series for the twisted Hodge numbers of $(\Hilb^nS, L_n)$:

\begin{theorem}
	\label{thm:main-numbers}
	Notation is as in \Cref{thm:main-cohomology}.
	We have the following equality of generating series in three variables:
	\begin{equation}
	\label{eqn:main-numbers}
	\sum_{n\geq 0}\sum_{p,q=0}^{2n}h^{p,q}(\Hilb^{n}{S},{L}_n)x^py^qt^n
	=
	\prod_{k\ge 1}\prod_{p, q=0}^2\left( 1-(-1)^{p+q}x^{p+k-1}y^{q+k-1}t^k\right)^{-(-1)^{p+q}h^{p,q}(S,{L}^{\otimes k})}.
	\end{equation}
	%where $h^{p,q}(\Hilb^{n}{S},{L}_n)=\dim H^q(\Hilb^nS, \Omega^p_{\Hilb^nS}\otimes L_n)$ and $h^{p,q}(S, L^{\otimes k})=\dim H^q(S, \Omega^p_S\otimes L^{\otimes k})$ are twisted Hodge numbers.
\end{theorem}

\begin{rmk}[Boissi\`ere's conjecture]
	The generating series \eqref{eqn:main-numbers} for the twisted Hodge numbers of the pair $(\Hilb^n(S), L_n)$, when $S$ is projective, was conjectured in Belmans--Fu--Krug \cite[Conjecture E]{BFK} as an amendment\footnote{Counter-examples to the original generating series of Boissi\`ere were found in \cite[Appendix B]{Hayashi}, and in \cite[Example 5.7]{BFK}. The only difference with respect to Boissi\`ere's formula is that we added an exponent $k$ to $L$ on the right-hand side.} of Boissière's conjecture \cite[Conjecture~1]{MR2932167}. The initial clue for us in \cite{BFK} was the computation of the Hochschild homology with coefficients as mentioned above in \eqref{eqn:HHCoeff}, together with G\"ottsche--Soergel's formula \eqref{eqn:HodgeNumbers}.
\end{rmk}

We will prove the following analogous result for nested Hilbert schemes $\Hilb^{n, n+1}S$; see \Cref{sec:NestedHilb} for the definitions and notations.

\begin{theorem}
	\label{thm:main-nested-cohomology}
	Let $S$ be a smooth compact complex surface. Let $L, L'$ be two holomorphic line bundles on $S$. We have a canonical isomorphism of tri-graded vector spaces:
	\begin{align}
	\label{eqn:main-nested-cohomology}
	\begin{split}
	\bigoplus_{n\geq 0}\bigoplus_{p,q\geq 0} H^{p,q}(\Hilb^{n, n+1}S, \phi^*L_n\otimes \rho^*L')x^py^qt^n\cong
	&\left(\bigoplus_{n\geq 0}\bigoplus_{p,q\geq 0} H^{p,q}(\Hilb^nS, L_n)x^py^qt^n\right)\\
	&\otimes \left(\bigoplus_{j\geq 0}\bigoplus_{p,q= 0}^2H^{p,q}(S, L^{\otimes j}\otimes L') x^{p+j}y^{q+j}t^j\right),
	\end{split}
	\end{align}
	where $\phi\colon \Hilb^{n,n+1}S\to \Hilb^nS$, $\rho\colon \Hilb^{n,n+1}S\to S$ are the natural morphisms. 
	
	More succinctly, 
	\begin{equation}
	\bigoplus_{n\geq 0} H^{\#,\star}(\Hilb^{n, n+1}S,\phi^* L_n\otimes \rho^*L')t^n\cong    \left(\bigoplus_{n\geq 0} H^{\#,\star}(\Hilb^nS, L_n)t^n\right)\otimes  \left(\bigoplus_{j\geq 0}H^{\#,\star}(S, L^{\otimes j}\otimes L')[-j, -j]t^j\right).
	\end{equation}
	
	Taking dimensions, we get the following generating series:
	\begin{equation}
	\label{eqn:main-HodgeNumbers-nested}
	\begin{split}
	&\sum_{n\geq 0}\sum_{p,q\geq 0} h^{p,q}(\Hilb^{n, n+1}S, \phi^*L_n\otimes \rho^*L')x^py^qt^n\\
	&= \left( \prod_{k\ge 1}\prod_{p, q=0}^2\left( 1-(-1)^{p+q}x^{p+k-1}y^{q+k-1}t^k\right)^{-(-1)^{p+q}h^{p,q}(S,{L}^{\otimes k})}\right)\cdot \left(\sum_{j\geq 0}\sum_{p,q= 0}^2h^{p,q}(S, L^{\otimes j}\otimes L')x^{p+j}y^{q+j}t^j\right).
	\end{split}
	\end{equation}
\end{theorem}
In the above statement, we used twisted Hodge groups of $\Hilb^nS$  with values in $L_n$ as building blocks on the right-hand side. Combining with \Cref{thm:main-cohomology}, one can express everything in terms of twisted Hodge groups of $S$ with values in powers of $L$.

\subsection{Application I: deformation theory of Hilbert schemes}
One of our motivations to study the \textit{twisted} Hodge groups is to understand the deformation theory of 
$\Hilb^nS$. Indeed, the relevant cohomology groups 
\begin{equation}
H^q(\Hilb^nS, T_{\Hilb^nS})\cong H^{2n-1, q}(\Hilb^nS, \omega_n^\vee)
\end{equation}
are twisted Hodge groups, and taking $L= \omega_S^\vee=\wedge^2T_S$ in \Cref{thm:main-cohomology} allows us to compute them and thus provides information on the formal deformation theory of $\Hilb^nS$, for any compact complex surface $S$. 

More precisely, we have the following result, whose proof is given in \Cref{sec:Deformation}, together with some examples.

\begin{theorem}[Deformation theory of $\Hilb^nS$]
	\label{thm:main-deformation}
	Let $S$ be a compact complex surface. For any $q\in \mathbb{N}$, we have a canonical isomorphism 
	\begin{align}
	\label{eqn:main-Cohomology-Tangent}
	\begin{split}
	H^q(\Hilb^nS, T_{\Hilb^nS})\cong  H^q(S^n, T_{S^n})^{\mathfrak{S}_n}
	& \oplus H^{q-1}(S^{(n-2)}, \mathscr{O})\otimes H^0(S, \wedge^2T_S)\\
	& \oplus H^{q-2}(S^{(n-2)}, \mathscr{O})\otimes H^1(S,  \wedge^2T_S)\\
	& \oplus H^{q-3}(S^{(n-2)}, \mathscr{O})\otimes H^2(S,   \wedge^2T_S).
	\end{split}
	\end{align}
	In particular, if $S$ is connected and $n\geq 2$, we have canonical isomorphisms:
	\begin{align}
	\label{eqn:H^0T}     H^0(\Hilb^nS, T_{\Hilb^nS})\cong& H^0(S, T_S);\\
	\label{eqn:H^1T}       H^1(\Hilb^nS, T_{\Hilb^nS})\cong &  
	H^1(S, T_S)\;\oplus\; H^0(S, T_S)\otimes H^1(S, \mathscr{O}_S)\;\oplus\;  H^0(S,  \wedge^2T_S);\\
	\label{eqn:H^2T}         \begin{split}
	H^2(\Hilb^nS, T_{\Hilb^nS})\cong   &
	H^2(S, T_S)\;\oplus\; H^1(S, T_S)\otimes H^1(S, \mathscr{O}_S) \\
	&\oplus\; H^0(S, T_S)\otimes H^2(S, \mathscr{O}_S)\;\oplus \; H^0(S, T_S)\otimes \wedge^2H^1(S, \mathscr{O}_S)\\
	&\oplus\; H^1(S , \mathscr{O}_S)\otimes H^0(S, \wedge^2T_S)\;\oplus\;  H^1(S, \wedge^2T_S).
	\end{split}
	\end{align}
	% Moreover, via the decompositions \eqref{eqn:H^1T} and \eqref{eqn:H^2T}, the Schouten--Nijenhuis bracket 
	% \begin{equation}
	%     [-,-]\colon  H^1(\Hilb^nS, T_{\Hilb^nS})\times  H^1(\Hilb^nS, T_{\Hilb^nS}) \to H^2(\Hilb^nS, T_{\Hilb^nS})
	% \end{equation}
	% is given by the Schouten--Nijenhuis bracket on $S$:
	% \begin{equation}
	%         [-,-]\colon H^i(S, \wedge^jT_S)\times  H^{i'}(S, \wedge^{j'}T_S)\to H^{i+i'}(S, \wedge^{j+j'-1}T_S).
	% \end{equation}
\end{theorem}

\begin{rmk}
	The relation  \eqref{eqn:H^0T} says that $S$ and $\Hilb^nS$ have the same infinitesimal automorphisms. This was first proved by Boissi\`ere \cite[Corollaire 1]{MR2932167} and rediscovered in \cite[Corollary 5.1]{BFK} by an alternative argument. Our proof here is identical to Boissi\`ere's.
	
	The relation \eqref{eqn:H^1T} describes the tangent space of the deformation space of $\Hilb^nS$. The fact that both sides of  \eqref{eqn:H^1T} have the same dimension was already proved in Belmans--Fu--Krug \cite[Corollary B]{BFK}, when $S$ is projective, using non-commutative methods, but no canonical isomorphism was provided there, due to the use of a cancellation argument. Here we not only construct a canonical isomorphism, but also extend the result to all compact complex surfaces. The isomorphism \eqref{eqn:H^1T} recovers as special cases the results of Fantechi \cite[Theorem 0.1, Theorem 0.3]{Fantechi-Deformation-1995} and Hitchin \cite[\S 4.1]{Hitchin-Deformation-2012}; see \Cref{sec:Deformation} for the statements of their results.
	
	The relation \eqref{eqn:H^2T} computes the obstruction space of the deformation theory of $\Hilb^nS$. 
	%We point out that   \eqref{eqn:H^1T} and  \eqref{eqn:H^2T}, together with the description of the Schouten--Nijenhuis bracket, grasps the most basic part of the formal deformation theory of $\Hilb^nS$ as complex manifold.  For example, the Kuranishi space $\Def(\Hilb^nS)$, as a germ of analytic space, is determined up to quadratic approximation by \Cref{thm:main-deformation}.
\end{rmk}

\begin{rmk}[Schouten--Nijenhuis bracket]
	Recall that for a compact complex manifold $X$, we have the Schouten--Nijenhuis bracket \cite{FrolicherNijenhuis, Nijenhuis-JacobiIdentities} 
	\begin{equation}
	[-,-]\colon H^q(X, \wedge^pT_X)\otimes H^{q'}(X, \wedge^{p'}T_X)\to H^{q+q'}(X,\wedge^{p+p'-1}T_X).
	\end{equation}
	A necessary condition for an infinitesimal deformation direction $\xi\in H^1(X, T_X)$ to be unobstructed is that $[\xi, \xi]=0$ in $H^2(X,T_X)$. In the case of Douady space of complex surfaces, we expect that via the decompositions \eqref{eqn:H^1T} and \eqref{eqn:H^2T}, the Schouten--Nijenhuis bracket 
	\begin{equation}
	[-,-]\colon  H^1(\Hilb^nS, T_{\Hilb^nS})\times  H^1(\Hilb^nS, T_{\Hilb^nS}) \to H^2(\Hilb^nS, T_{\Hilb^nS})
	\end{equation}
	is given by the Schouten--Nijenhuis bracket on $S$:
	\begin{equation}
	[-,-]\colon H^i(S, \wedge^jT_S)\times  H^{i'}(S, \wedge^{j'}T_S)\to H^{i+i'}(S, \wedge^{j+j'-1}T_S).
	\end{equation}
	Assuming this, the Kuranishi space $\Def(\Hilb^nS)$, as a germ of analytic space, is determined up to quadratic approximation by \eqref{eqn:H^1T} and \eqref{eqn:H^2T}. We plan to pursue this direction in a follow-up work.
\end{rmk}

\subsection{Application II: extending formulas beyond the K\"ahler setting}

\Cref{thm:main-cohomology} and \Cref{thm:main-numbers} recover various aforementioned results by specializing, and therefore proving them for any compact complex surface $S$, without any algebraicity or K\"ahlerness assumption.

Setting $L=\mathscr{O}_S$ in \eqref{eqn:main-numbers}:
\begin{cor}
	\label{cor:HodgeNumbers}
	G\"ottsche--Soergel's formula \eqref{eqn:HodgeNumbers} for Hodge numbers of Douady spaces holds for any compact complex surface $S$:
	\begin{equation*}
	\label{eqn:HodgeNumbers2}
	\sum_{n\geq 0}\mathbb{E}_{\Hilb^nS}(x,y)t^n:= \sum_{n\geq 0}\sum_{p, q\geq 0} h^{p,q}(\Hilb^nS)x^py^qt^n=\prod_{k\geq 1}\prod_{p,q=0}^2 (1-(-1)^{p+q}x^{p+k-1}y^{q+k-1}t^k)^{-(-1)^{p+q}h^{p,q}(S)}.
	\end{equation*} \end{cor}

\begin{cor}
	\label{cor:Frolicher}
	For any compact complex surface $S$ and any positive integer $n$, the Fr\"olicher spectral sequence for the Douady space $\Hilb^nS$
	\begin{equation}
	E_1^{p,q}=H^q(\Hilb^nS, \Omega_{\Hilb^nS}^p) \Rightarrow H^{p+q}(X, \CC)
	\end{equation}
	degenerates at $E_1$-page.
\end{cor}

Similarly, setting $L=L'=\mathscr{O}_S$ in \eqref{eqn:main-HodgeNumbers-nested} in \Cref{thm:main-nested-cohomology}, we obtain the following extension of Cheah's formula \cite[P.485]{Cheah-VirtualHodgePolynomial}  for all compact complex surfaces:
\begin{cor}
	\label{cor:HodgeNumbers-Nested}
	For any compact  complex surface $S$, we have the following equality of Hodge polynomials:
	\begin{equation}
	\label{eqn:HodgeNumbers-Nested}
	\sum_{n\geq 0} \mathbb{E}_{\Hilb^{n,n+1}S}(x,y)t^n=\left(\sum_{n\geq 0} \mathbb{E}_{\Hilb^nS}(x,y)t^n\right) \cdot\mathbb{E}_S(x,y)\cdot \frac{1}{1-xyt}.
	\end{equation}
\end{cor}

\begin{rmk}[Hodge numbers in non-K\"ahler situation]
	For a non-K\"ahler compact complex surface $S$, the Douady spaces $\Hilb^nS$  and $\Hilb^{n,n+1}S$ are not K\"ahler for $n>0$. 
	To the best of my knowledge, \Cref{cor:HodgeNumbers} and \Cref{cor:HodgeNumbers-Nested} are the first time that these Hodge numbers are computed in this generality in the literature. Indeed, G\"ottsche--Soergel \cite{GottscheSoergel} and Cheah \cite{Cheah-VirtualHodgePolynomial} needed the K\"ahlerness assumption in their argument to recover Hodge numbers from the Hodge structures on the cohomology of $\Hilb^nS$ and $\Hilb^{n,n+1}S$. (Nevertheless, the G\"ottsche formula \eqref{eqn:BettiNumbers} for Betti numbers is established in the general complex analytic setting by de Cataldo--Migliorini \cite[Theorem 5.2.1]{deCataldoMigliorini-DouadySpace-2000}.) 
	
	Here is an example: for an Inoue or a Hopf surface $S$ (see \cite[Chapter V, \S\S 18, 19]{BHPV-SurfaceBook}), its Hodge polynomial is $1+y+x^2y+x^2y^2$ (not symmetric in $x$ and $y$), and \eqref{eqn:HodgeNumbers2} gives the Hodge polynomial
	\begin{equation}
	\label{eqn:HodgeNumbers-InoueHopf}
	\sum_{n\geq 0}\sum_{p, q\geq 0} h^{p,q}(\Hilb^nS)x^py^qt^n=\prod_{k\geq 1}\frac{(1+x^{k-1}y^kt^k)(1+x^{k+1}y^kt^k)}{(1-x^{k-1}y^{k-1}t^k)(1-x^{k+1}y^{k+1}t^k)},
	\end{equation}
	again, not symmetric in $x$ and $y$. Below are the Hodge diamonds\footnote{These Hodge diamonds are produced using Pieter Belmans' \textit{Hodge diamond cutter} package \cite{hodge-diamond-cutter}.} of $S$, $\Hilb^2S$ and $\Hilb^3S$:
	\[
	\begin{smallmatrix}
	&& 1 &&\\
	& \color{blue} 1 && \color{blue} 0 &\\
	0 && 0 && 0\\
	& {\color{blue} 0} && \color{blue} 1 &\\
	&& 1 &&
	\end{smallmatrix} \quad\quad
	\begin{smallmatrix}
	&&&& 1 &&&&\\
	&&&  \color{blue}1 &&  \color{blue}0 &&&\\
	&& 0 && 1 && 0&&\\
	& 0 && 1 && 1 && 0&\\
	0&& 0 && 2 && 0 && 0\\
	& 0 && 1 && 1 && 0&\\
	&& 0 && 1 && 0&&\\
	&&& \color{blue} 0 && \color{blue} 1 &&&\\
	&&&& 1 &&&&
	\end{smallmatrix}
	\quad\quad
	\begin{smallmatrix}
	&&&&&& 1 &&&&&&\\
	&&&&& \color{blue} 1 && \color{blue} 0 &&&&&\\
	&&&& 0 && 1 && 0&&&&\\
	&&& 0 && \color{blue}2 && \color{blue}1 && 0&&&\\
	&&0&& \color{blue}1 && 3 && \color{blue}0 && 0&&\\
	&0 &&0&&2&&2&&0&&0&\\
	0&&0&&0&&4&&0&&0&&0\\
	&0 &&0&&2&&2&&0&&0&\\
	&&0&& \color{blue}0 && 3 && \color{blue}1 && 0&&\\
	&&& 0 && \color{blue}1 && \color{blue}2 && 0&&&\\
	&&&& 0 && 1 && 0&&&&\\
	&&&&& \color{blue} 0 && \color{blue} 1 &&&&&\\
	&&&&&& 1 &&&&&&
	\end{smallmatrix}
	\]
	For a secondary Kodaira surface, the Hodge polynomials of its Douady spaces are given by  \eqref{eqn:HodgeNumbers-InoueHopf} with $x$ and $y$ switched, hence their Hodge diamonds are obtained from the above ones by reflecting with respective to the middle vertical line. For more examples of Hodge diamonds of Douady spaces of non-K\"ahler surfaces, we recommend the package \cite{hodge-diamond-cutter} which implements \Cref{thm:main-numbers} and \Cref{cor:HodgeNumbers}.
	
	On a different note, we mention that the G\"ottsche--Soergel formula of Hodge numbers can fail in positive characteristics, as is shown by Srivastava \cite{Srivastava} for Hilbert schemes of supersingular Enriques surfaces in characteristic 2.
\end{rmk}

\medskip
Setting $y=-1$ and renaming $x$ by $-y$ in \eqref{eqn:main-numbers}, we get:
\begin{cor}
	G\"ottsche's formula \eqref{eqn:ChiyGenusCoeff} for the refined $\chi_y$-genera of Douady spaces holds for any compact complex surface $S$:
	\begin{equation}
	\label{eqn:ChiyGenusCoeff2}
	\sum_{n\geq 0} \chi_{-y}(\Hilb^nS, L_n)t^n= \prod_{k\geq 1}\prod_{p\geq 0} \left(1-y^{p+k-1}t^k\right)^{-(-1)^p\chi(S,\Omega_S^p\otimes L^{\otimes k})}=\exp\left(\sum_{m\geq 1}\frac{t^m}{m}\sum_{k\geq 1}(ty)^{(k-1)m}\chi_{-y^m}(S, L^{\otimes k})\right).
	\end{equation}
\end{cor}
\begin{rmk}
	\begin{enumerate}[label=\emph{\roman*})]
		\item The second equality in \eqref{eqn:ChiyGenusCoeff2} is purely elementary; see \cite[Remark 5.12]{BFK}.
		\item Our proof of \eqref{eqn:ChiyGenusCoeff2} does not use the Riemann--Roch theorem. To recover G\"ottsche's original formula \cite[Corollary 1.2]{Gottsche-RefinedVerlinde-2020}, one only needs to apply the Riemann--Roch formula for the surface $S$ (but not for $\Hilb^nS$).
		\item Of course G\"ottsche--Soergel's formula \eqref{eqn:ChiyGenera} for $\chi_y$-genera  is also recovered and extended to the non-K\"ahler setting, but this was already proved by Ellingsrud--G\"ottsche--Lehn \cite[Theorem 1.3]{EllingsrudGottscheLehn} since their argument of cobordism works without the projectivity (or K\"ahlerness) assumption.
	\end{enumerate}
\end{rmk}

\medskip
Setting  $x=y^{-1}$ in \eqref{eqn:main-cohomology} and using the Hochschild--Kostant--Rosenberg isomorphism, we obtain the following consequence. The proof is given at the end of \Cref{sec:Proof}.
\begin{cor}
	\label{cor:BFK-HH}
	Belmans--Fu--Krug's formula \eqref{eqn:HHCoeff} for Hochschild homology with coefficients holds for any compact complex surface $S$.
	In particular, the same is true for the formulas for Hochschild homology \eqref{eqn:HH}, Hochschild cohomology  \eqref{eqn:HH*}, and Hochschild--Serre cohomology \eqref{eqn:HS}.
\end{cor}

%\begin{rmk}[Relations to known results]
%     Let $S$ be projective (or at least K\"ahler). In \eqref{eqn:main-cohomology} and \eqref{eqn:main-numbers}, 
%    \begin{itemize}
%       \item setting $L=\mathscr{O}_S$, we get G\"ottsche--Soergel's formula for Hodge numbers \eqref{eqn:HodgeNumbers},
%        hence G\"ottsche's formula \eqref{eqn:BettiNumbers} for Betti numbers;
%        \item setting $y=-1$ and renaming $x$ by $-y$, we obtain G\"ottsche's formula for refined $\chi_y$-genera \eqref{eqn:ChiyGenusCoeff}, hence \eqref{eqn:ChiyGenera} in particular;
%        \item setting $x=y^{-1}$, and using the Hochschild--Kostant--Rosenberg isomorphism, we deduce Belmans--Fu--Krug's formula \eqref{eqn:HHCoeff} for Hochschild homology with coefficients, hence in particular \eqref{eqn:HH} , \eqref{eqn:HH*} and  \eqref{eqn:HS}, by further specializing to $L=\mathscr{O}_S$, $\omega_S^\vee$ and $\omega_S^{k-1}$ respectively.
%    \end{itemize}
%\end{rmk}

\medskip

\subsection{Method of proof}
In \cite{BFK}, our proof of \eqref{eqn:HHCoeff} is essentially "non-commutative" in the sense that we used the equivalence of derived categories established by Bridgeland--King--Reid \cite{BridgelandKingReid} to reduce the computation of the Hochschild homology of $\Hilb^nS$ with values in $L_n$ to the Hochschild homology of the symmetric quotient stack $[S^n/\mathfrak{S}_n]$ with values in $L^{\boxtimes n}$ (equipped with $\mathfrak{S}_n$-linearization).

However in this paper, our method is more classical: it relies on a study of the Hilbert--Chow morphism (or Douady--Barlet morphism) from the Hilbert scheme to the symmetric power $$\pi\colon \Hilb^nS \to S^{(n)}.$$ The key point is to exploit the "relative Hodge theory" of $\pi$, namely, the Hodge modules naturally produced via $\pi$ by Saito's decomposition theorem. In this sense, this paper is a natural continuation of the seminal work of G\"ottsche--Soergel \cite{GottscheSoergel}. A crucial intermediate result is \Cref{prop:PushforwardOfOmega}, whose statement does not involve Hodge modules, and it can be of independent interest.

\begin{rmk}[Non-compact case]
	We would like to point out that for any smooth complex surface $S$, not necessarily algebraic or K\"ahler, or even compact, the Hilbert--Chow morphism $\pi$ is always a \textit{projective} morphism. This is the basic reason why our results are valid without the K\"ahlerness assumption. Moreover, compactness condition is nowhere used in any proof in this paper, hence all our results remain valid if we replace the compactness assumption by the assumption that all the numbers and dimensions appearing in the formulas are finite; the latter finite dimensionality is indeed guaranteed by the compactness condition by the classical theorem of Cartan--Serre \cite{CartanSerre-Finitude} and Grauert \cite{Grauert-IHES}.
\end{rmk}

\subsection*{Acknowledgment}: I am very grateful to Wanchun Shen, Chuanhao Wei and Ruijie Yang for helpful discussions on Hodge modules. I thank Pieter Belmans, Samuel Boissi\`ere, Andreas Krug, and Claire Voisin for their comments and questions. I am benefited from my visit to Shanghai Center for Mathematical Sciences (SCMS) in summer 2024. I want to thank Zhi Jiang and Zhiyuan Li for organizing the very nice summer school and for their hospitality during my stay. My research is supported by the University of Strasbourg Institute for Advanced Study (USIAS), and by the Agence Nationale de la Recherche (ANR) under projects ANR-20-CE40-0023 and ANR-24-CE40-4098. 

\section{Hodge modules}

This section recalls some basic theory of Hodge modules that we need. We refer to the original paper of M. Saito \cite{Saito88-MHPol} for more details. Roughly speaking, Hodge modules are generalizations of variations of Hodge structures by replacing local systems by more general perverse sheaves and replacing flat connections by more general $\D$-modules, such that these two data are related by the Riemann--Hilbert correspondence established by Kashiwara \cite{Kashiwara-RHProblem-1984} and Mebkhout \cite{Mebkhout-RH-1, Mebkhout-RH-2}.

\subsection{The notion of Hodge modules}
Let $X$ be a complex manifold of dimension $d$.
For an integer $w$, we denote by $\HM(X, w)$ the abelian category of (pure) Hodge modules of weight $w$ on $X$. The subcategory of weight-$w$ polarizable (pure) Hodge modules is denoted by $\HM^\p(X, w)$.

A Hodge module on $X$ is a filtered regular holonomic $\D_X$-module with a rational structure that is subject to some conditions. These conditions are based on Schmid's work \cite{Schmid} on degenerations of variations of Hodge structures. We refer to the original paper of M. Saito \cite{Saito88-MHPol} as well as the surveys \cite{saito1989introduction} and \cite{Schnell-Sanya} for the precise definition. Let us simply recall here  the underlying basic structure (see for example \cite[Definition 7.1.]{Schnell-Sanya}): \textit{a filtered regular holonomic $\D_X$-module with a rational structure} is the datum $M=(\mathbb{M}, \mathcal{M}, F_\bullet\mathcal{M})$ consisting of  a perverse sheaf with $\QQ$-coefficients $\mathbb{M}\in \Perv_\QQ(X)$, a regular holonomic \textit{right} $\D_X$-module $\mathcal{M}$, and an increasing exhaustive filtration $F_\bullet $ of $\mathcal{M}$ by coherent $\mathscr{O}_X$-modules that is compatible with the natural filtration on $\mathscr{D}_X$, 
\begin{equation}
F_i\mathcal{M}\cdot F_j\mathscr{D}_X\subset F_{i+j}\mathcal{M},
\end{equation}
and is \textit{good} in the sense that there exists $i$ such that for all $j\geq 0$
\begin{equation*}
F_i\mathcal{M}\cdot F_j\D_X=F_{i+j}\mathcal{M},
\end{equation*}
such that the complexification $\mathbb{M}_\CC\in \Perv_\CC(X)$ and $\mathcal{M}$ are related by the Riemann--Hilbert correspondence, that is, $\mathbb{M}_{\CC}$  is isomorphic to the de Rham complex associated with the right $\D_X$-module $\mathcal{M}$:
\begin{equation}
\mathbb{M}_\CC\cong \DR(\mathcal{M}).
\end{equation}

Recall that the de Rham complex is defined as the following complex living in degrees $-d, \dots, 0$:
\begin{equation}
\DR(\mathcal{M}):=\left[\mathcal{M}\otimes \bigwedge^dT_X\to \mathcal{M}\otimes \bigwedge^{d-1}T_X\to\cdots \to \mathcal{M}\otimes T_X\to \mathcal{M}\right] .
\end{equation}
It is equipped with the following filtration: for any $p\in \ZZ$,
\begin{equation}
F_p\DR(\mathcal{M}):=\left[F_{p-d}\mathcal{M}\otimes \bigwedge^dT_X\to F_{p-d+1}\mathcal{M}\otimes \bigwedge^{d-1}T_X\to\cdots \to F_{p-1}\mathcal{M}\otimes T_X\to F_p\mathcal{M}\right] .
\end{equation}

The associated graded pieces are  the following complexes of $\mathscr{O}_X$-modules living in degrees $-d, \dots, 0$:
\begin{equation}
\gr^F_p\DR(\mathcal{M})=\left[\gr^F_{p-d}\mathcal{M}\otimes \bigwedge^dT_X\to \gr^F_{p-d+1}\mathcal{M}\otimes \bigwedge^{d-1}T_X\to\cdots \to \gr^F_{p-1}\mathcal{M}\otimes T_X\to \gr^F_p\mathcal{M}\right].
\end{equation}

\begin{rmk}[Hodge modules on singular spaces]
	\label{rmk:HMonSingularSpace}
	Given a possibly singular complex analytic space $X$, Hodge modules on $X$ are defined via some ambient complex manifold $Y$ into which $X$ is embedded (such embedding exists for instance when $X$ is quasi-projective; otherwise one uses local embeddings and glues constructions as in Saito \cite{Saito90-MHM}). More precisely, given such an embedding $X\hookrightarrow Y$, define the category of Hodge modules on $X$ of weight $w$
	\begin{equation}
	\HM(X, w)
	\end{equation}
	to be the subcategory of Hodge modules of weight $w$ on $Y$ with support in $X$. This definition is independent of the choice of the ambient manifold $Y$ (the key ingredient being Kashiwara's equivalence \cite{Kashiwara-MemoireSMF}); see \cite[\S 14]{Schnell-Sanya}.
\end{rmk}
\begin{rmk}
	\label{rmk:grDRwell-defined}
	Given a Hodge module $M=(\mathbb{M}, \mathcal{M}, F_\bullet\mathcal{M})\in \HM(X)$ on a singular analytic space $X$, as in \Cref{rmk:HMonSingularSpace}, $M$ is defined via an ambient complex manifold $Y$. Although $\mathcal{M}$ is a $\D$-module on $Y$, the associated graded pieces $\gr^F_{\bullet} \DR(\mathcal{M})$ are well-defined objects in the derived category of coherent $\mathscr{O}_X$-modules, and are independent of the embedding of $X$ into $Y$, by Schnell \cite[Lemma 7.3]{Schnell-SaitoVanishing-2016}.
\end{rmk}

\subsection{Basic examples}

In order to fix notations, we introduce some examples of Hodge modules that will be used in the paper.
\subsubsection{Constant Hodge module}
Given a complex manifold $X$ of dimension $d$, we have the following basic example of polarizable pure Hodge module of weight $d$:
\begin{equation}\label{eqn:TrivialHM}
\QQ^H_X[d]:=(\QQ_X[d], \omega_X, F_\bullet) \in \HM^\p(X, d),
\end{equation}
called the \textit{constant Hodge module} on $X$,
where $\omega_X$ denotes the canonical bundle viewed as a right $\D_X$-module, the filtration is determined by requiring $F_{-d}\omega_X=\omega_X$ and $F_{-d-1}\omega_X=0$.
The de Rham complex of the underlying right $\mathscr{D}_X$-module $\omega_X$, living in degrees $-d, \dots, 0$,  is the $d$-shift of the classical de Rham complex:
\begin{equation}
\DR(\omega_X)=\left[\mathscr{O}_X\to \Omega^1_X\to \cdots \to \Omega^d_X\right]=\Omega_X^\bullet[d].
\end{equation}

The induced filtration is given as follows:
\begin{align*}
F_0\DR(\omega_X)&=\Omega^\bullet_X[d] \\
F_{-d-1}\DR(\omega_X)&=0 ;\\
F_{-p}\DR(\omega_X)&=\left[\Omega^p_X\to \cdots \to \Omega^d_X\right] \text{ for } 0\leq p\leq d.
\end{align*}
In particular,  for any $p\in \ZZ$, we have
\begin{equation}
\label{eqn:grDRomega}
\gr^F_{-p}\DR(\omega_X)=\Omega_X^p[d-p].
\end{equation}
\subsubsection{Variations of Hodge structure}
More generally, let $V=(\mathbb{V}, \mathcal{V}, \nabla, F^\bullet)\in \VHS^\p(X, w)$ be a polarizable variation of (pure) Hodge structure of weight $w$ over a $d$-dimensional complex manifold  $X$, where $\mathbb{V}$ is the underlying local system, $(\mathcal{V}, \nabla)$ is the flat connection and $F^\bullet$ is the Hodge filtration. We can naturally associate to $V$ a polarizable Hodge module of weight $w+d$ as follows:
\begin{equation}
V^H[d]:=(\mathbb{V}[d], \omega_X\otimes_{\mathscr{O}_X}\mathcal{V}, F_{\bullet})\in \HM^\p(X, w+d),
\end{equation}
where the filtration is defined as $F_p(\omega_X\otimes_{\mathscr{O}_X}\mathcal{V})=\omega_X\otimes_{\mathscr{O}_X}F^{-p-d}\mathcal{V}$ for any $p\in \ZZ$. Note that $\mathcal{V}$ is naturally a left $\D_X$-module, and after tensoring with $\omega_X$ we convert it into a right $\D_X$-module. 

\subsubsection{IC Hodge module}
Given an irreducible complex analytic space $X$ of dimension $d$, let $V=(\mathbb{V}, \mathcal{V}, \nabla, F^\bullet)\in \VHS^\p(U,w)$ be a polarizable variation of Hodge structure of weight $w$ on a non-empty smooth Zariski open subset $U$ of $X$. One main achievement in Saito \cite{Saito90-MHM} is that there is a canonical extension of $V^H[d]\in \HM^\p(U, w+d)$  to a polarizable Hodge module on $X$, denoted by
\begin{equation}
\ICH_X(V)\in \HM^\p(X, w+d),
\end{equation}
whose underlying perverse sheaf is $\IC_X(\mathbb{V})$, the intermediate extension of $\mathbb{V}[d]$ from $U$ to $X$.

In particular, viewing $\QQ$ as the trivial variation of Hodge structure on a smooth Zariski open subset of $X$, the Hodge module $\ICH_X(\QQ)\in \HM^\p(X, d)$ has as underlying perverse sheaf $\IC_X=\IC_X(\QQ)$, and its underlying filtered $\D_X$-module is often denoted simply by $\ICH_X$.

More generally, let $Z$ be an irreducible subvariety of a complex analytic space $X$. Let $i\colon Z\to X$ be the closed immersion and let $d_Z$ be the dimension of $Z$. Let $V\in \VHS^\p(U,w)$ be a polarizable variation of Hodge structure of weight $w$ on a smooth Zariski open dense subvariety $U$ of  $Z$. By Saito \cite{Saito90-MHM}, we have the following polarizable Hodge module on $X$ with $i_*\IC_Z(\mathbb{V})$  as underlying perverse sheaf:

\begin{equation}
i_*\ICH_Z(V)\in \HM^\p(X,w+d_Z).
\end{equation}
Note that $i_*\ICH_Z(V)$ has strict support $Z$, and the structure theorem of Saito \cite{Saito90-MHM} says that every object in $\HM^\p(X)$ with strict support $Z$ is of this form:
\begin{equation}
\HM^\p_Z(X)=\left\{  i_*\ICH_Z(V)~|~ V\in \VHS^\p(U), U\subset Z \text{  open smooth }\right\}.
\end{equation}
Moreover, Saito's strict support decomposition theorem \cite{Saito88-MHPol} (see also \cite[Theorem 15.1]{Schnell-Sanya}) says that as abelian categories
\begin{equation}
\HM^\p(X)
\cong \bigoplus_{Z\subset X}\HM^\p_Z(X),
\end{equation}
with $Z$ running through all integral subvarieties of $X$.

\subsection{Constant Hodge module and Du Bois complex}
For an irreducible complex algebraic variety $X$ of dimension $d$, the \textit{Du Bois complex} of $X$, denoted by $\underline\Omega^\bullet_X$, constructed by Deligne and Du Bois \cite{DuBois}, is a filtered complex of sheaves (defined up to quasi-isomorphism).  For any $p\in \mathbb{N}$, the Du Bois complex of $p$-forms is defined as its associated graded pieces:, $$\underline\Omega_X^p:=\gr^F_{-p}\underline\Omega^\bullet_X [p]:=\gr_F^{p}\underline\Omega^\bullet_X [p].$$
Note that this is a well-defined object in the derived category of coherent $\mathscr{O}_X$-modules \cite[Proposition 7.24]{PetersSteenbrink-MHSBook}. 
If $X$ is smooth, this recovers the usual holomorphic de Rham complex: $\Omega_X^\bullet\cong \underline\Omega^{\bullet}_X$ with the stupid filtration, and $\Omega_X^p\cong \underline\Omega^{p}_X$ in $\operatorname{D^b_{coh}}(X)$.  For the construction and basic properties about the Du Bois complex, we refer to the original source \cite{DuBois}  as well as to \cite{Steenbrink-MHSonVanishingCohomology-1977}  and \cite[\S 7.3]{PetersSteenbrink-MHSBook}. 

Thanks to Saito \cite{Saito-MixedHodgeComplex-DuBois-2000}, the Du Bois complex is closely related to Hodge modules.  Recall that the \textit{constant Hodge module} on a possibly singular complex analytic space $X$, denoted by $\QQ_X^H[d]$, which in general is a mixed Hodge module, is defined as the inverse image of the trivial Hodge module $\QQ$ on a point via the structural map $X\to \operatorname{pt}$. We denote the underlying filtered $\D_X$-module by the same notation $\QQ_X^H[d]$.

Saito \cite[Theorem 0.2]{Saito-MixedHodgeComplex-DuBois-2000} gives a precise relation between the Du Bois complex and the de Rham complex of the constant Hodge module. Let us only give the following  characterization on the associated graded pieces:
\begin{equation}
\label{eqn:grDRIC-2}
\underline\Omega_X^p\cong \gr^F_{-p}\DR(\QQ^H_X[d])[p-d].
\end{equation}

%In terms of pure Hodge modules on a smooth ambient variety: 
%
%\begin{prop}[Musta\c t\u a--Popa {\cite[Proposition 5.5]{MustataPopa-LocalCohomology-Pi-2022}}]
%   Let $i\colon X\hookrightarrow Y$ be a closed immersion of an algebraic variety $X$ into a smooth algebraic variety $Y$ of dimension $n$. Then for any integer $p$, we have an isomorphism in $\operatorname{D^b_{coh}}(X)$:

%\begin{equation}
%   \underline\Omega_X^p\cong \R\mathcal{H}om_{\mathscr{O}_Y}(\gr^F_{p-n} \DR_Y (i_*i^!\QQ_Y^H[n]), \omega_Y)[p].
%\end{equation}
%\end{prop}

In \cite[Section 5]{DuBois}, Du Bois identified, for normal algebraic varieties with finite quotient singularities (V-manifolds), his complex with the de Rham complex of reflexive differentials; see \Cref{sec:VManifolds} below for the terminology about V-manifolds.

\begin{theorem}[Du Bois {\cite[Th\'eor\`eme 5.3]{DuBois}}]
	\label{thm:DuBoisComplexVManifolds}
	For an algebraic V-manifold $X$, there is a canonical isomorphism between the Du Bois complex and the de Rham complex of reflexive differentials as objects in the derived category of filtered complexes of sheaves:
	\begin{equation}
	\underline\Omega_X^\bullet\cong \Omega_X^{[\bullet]}.
	\end{equation}
	In particular, $\Omega_X^{[p]}\cong \underline\Omega_X^p$ for any $p\in \ZZ$.
\end{theorem}

\subsection{Strictness of direct images}

The following theorem due to Saito on the strictness of direct image is the key ingredient which allows us to access the direct image of sheaves of differential forms:

\begin{theorem}[Saito {\cite[2.3.7]{Saito88-MHPol}}, see also {\cite[Theorem 28.1]{Schnell-Sanya}}]
	\label{thm:Strictness}
	Let $f\colon X \to Y$ be a projective morphism between complex manifolds. Let $M=(\mathbb{M},\mathcal{M}, F_\bullet \mathcal{M})$
	be a Hodge module  on $X$. Then for any $p\in \ZZ$,  we have isomorphisms
	\begin{equation}
	\R f_*\gr^F_p \DR (\mathcal{M})\cong \gr^F_p\DR (f_+\mathcal{M})\cong \bigoplus_{i\in \ZZ} \gr_p^F \DR(\mathcal{H}^if_+\mathcal{M})[-i].
	\end{equation}
\end{theorem}

\begin{rmk}[Singular target]
	\label{rmk:SingularTarget}
	Let us point out that the smoothness assumption on $Y$ in   \Cref{thm:Strictness} is not necessary: the same result holds true for $Y$ a complex analytic space that is embeddable into a complex manifold. This is probably well-known to experts. But since we will need to apply the theorem in this more general setting, let us give the argument here. Let $i\colon Y\hookrightarrow Z$ be an embedding of $Y$ into a complex manifold $Z$, and let $g:=i\circ f$ be the composed morphism $X\to Z$. Then by \Cref{thm:Strictness}, we have isomorphisms 
	\begin{equation}
	\label{eqn:IsomOnZ}
	i_*\R f_*\gr^F_p \DR (\mathcal{M})\cong \R g_*\gr^F_p \DR (\mathcal{M})\cong \gr^F_p\DR (g_+\mathcal{M}).
	\end{equation}  However, by construction, the $\D_Z$-module $g_+\mathcal{M}$ is supported on $Y$, hence by \cite[Lemma 7.3]{Schnell-SaitoVanishing-2016}, $\gr^F_p\DR (g_+\mathcal{M})$ is a well-defined (independent of $i$ and $Z$) complex of coherent $\mathscr{O}_Y$-modules up to quasi-isomorphism, namely, by definition: 
	\begin{equation}
	\gr^F_p\DR (g_+\mathcal{M})\cong i_*\gr^F_p\DR (f_+\mathcal{M}).
	\end{equation}
	Since $i_*$ is conservative, we conclude that $\R f_*\gr^F_p \DR (\mathcal{M})\cong \gr^F_p\DR (f_+\mathcal{M})$.
\end{rmk}

\subsection{Finite birational morphisms}

\begin{lemma}
	\label{lemma:BirationalFinite}
	Let $\nu\colon X\to Y$ be a finite birational morphism (e.g. normalization) between complex algebraic varieties (or bimeromorphic map between normal complex analytic spaces). Let $U$ be a Zariski-open subvariety of $X$ over which $\nu$ is an isomorphism. Let $V$ be a polarizable variation of Hodge structures over $U$. Then we have an isomorphism of Hodge modules over $Y$:
	\begin{equation}
	\nu_*\ICH_X(V)\cong \ICH_Y(V).
	\end{equation}
\end{lemma}
\begin{proof}
	Denote by $j\colon U\to X$ and $j'\colon U\to Y$ the open immersions, where $j'=\nu\circ j$.  Let $d=\dim(X)$. By definition, 
	\begin{equation}
	\ICH_X(V)=j_{!*}(V^H[d]):=\operatorname{Im}(\mathcal{H}^0j_!V^H[d]\to \mathcal{H}^0j_*V^H[d]).
	\end{equation}
	Here $\mathcal{H}^0$ is taken in the derived category of Hodge modules (corresponding to the perverse ${}^\p\mathcal{H}^0$ for the underlying complex of constructible sheaves).
	Since $\nu$ is finite, $\nu_*=\nu_!$ is t-exact, hence commutes with $\mathcal{H}^0$ and preserves kernels, cokernels and images. Therefore, 
	\begin{equation}
	\nu_*\ICH_X(V)\cong\operatorname{Im}(\mathcal{H}^0\nu_!j_!V^H[d]\to \mathcal{H}^0\nu_*j_*V^H[d])\cong \operatorname{Im}(\mathcal{H}^0j'_!V^H[d]\to \mathcal{H}^0j'_*V^H[d])=\ICH_Y(V).
	\end{equation}
\end{proof}

\section{V-manifolds and reflexive differentials}
\label{sec:VManifolds}

A \textit{V-manifold} is a normal complex analytic space with finite quotient singularities; see Satake's original papers for the generalities \cite{Satake-1956, Satake-1957}. For the Hodge theory of V-manifolds, we refer to Steenbrink \cite{Steenbrink-MHSonVanishingCohomology-1977} as the basic reference.

Let $X$ be a V-manifold of dimension $d$. According to Prill \cite[Proposition 6]{Prill}, each singular point $x\in X$ admits an open neighborhood $U_x$ with orbifold chart $(V_x, G_x, U_x\xrightarrow{\cong} V_x/G_x)$ with $G_x$ a finite subgroup of $\GL_d(\CC)$ and  $0\in V_x\subset \CC^d$ an open subset stable under $G_x$, such that the fixed loci for all $g\in G_x\backslash\{\id\}$  are of codimension $\geq 2$. The group $G_x$ is uniquely determined up to conjugation, called the \textit{local fundamental group}, or the \textit{stabilizer group}, at $x$. These orbifold charts determines a unique effective analytic Deligne--Mumford stack $\mathscr{X}$ with stacky locus of codimension $\geq 2$ and with underlying coarse moduli space $X$. The data of $X$ and $\mathscr{X}$ are thus equivalent data. 

\subsection{V-bundles}
We recall the notion of $V$-bundles (see for example \cite[Section 2]{Blache}). Keep the notations as above. A \textit{V-bundle} on a V-manifold $X$ is a reflexive coherent $\mathscr{O}_X$-module $F$ such that for any orbifold chart $(V_x, G_x, U_x\xrightarrow{\cong} V_x/G_x)$ as above, $\hat{F}=\varpi^*(F|_{U_x})^{\vee\vee}$ is a vector bundle on $V_x$, where $\varpi\colon V_x\to V_x/G_x\xrightarrow{\cong} U_x$ is the natural map. The reflexive sheaf $F$ can be recovered from the vector bundle $\hat{F}$ endowed with the natural $G_x$-linearization via invariant push-forward: $F|_{U_x}\cong (\varpi_*\hat{F})^{G_x}$.

In general, for a finite group $G$ acting on a complex manifold $V$ with fixed loci for all $g\in G_x\backslash\{\id\}$ of codimension $\geq 2$, let $U=V/G$ be the corresponding V-manifold and $\varpi\colon V\to U$ the natural map. There is an equivalence of categories between the category of V-bundles on $U$, and the category of vector bundles on $V$ with $G$-linearization given as follows:
\begin{align}
\label{eqn:V-bundle}
\begin{split}
\{\text{V-bundles on  } U \} &\xrightarrow{\cong} \{\text{Vector bundles on $V$ with $G$-linearization} \}\\
F &\mapsto \varpi^*(F)^{\vee\vee}
\end{split}
\end{align}
with inverse given by $\hat{F}\mapsto \varpi_*(\hat{F})^G$.
This also gives an equivalence of categories between vector bundles on $\mathscr{X}$ and V-bundles on $X$.

\subsection{Reflexive differentials}

\begin{definition}[Reflexive differentials]
	Let $X$ be a V-manifold of dimension $d$. For any $0\leq p\leq d$, we define the sheaf of  \textit{reflexive $p$-differentials} as
	\begin{equation}
	\Omega_X^{[p]}:= j_*\Omega^p_{X_{\operatorname{reg}}},
	\end{equation}
	where $j\colon X_{\operatorname{reg}}\to X$ is the open immersion of the smooth locus of $X$. Note that when $X$ is algebraic, $\Omega_X^{[p]}$ is nothing but the analytification of the reflexive hull of the K\"ahler $p$-differentials $\Omega_X^p$ and it is a V-bundle.
	
	One can form the de Rham complex of reflexive differentials $\Omega_X^{[\bullet]}$, living in degrees $\{0, \dots, d\}$. It is naturally equipped with the stupid filtration: $F_{-p}\Omega_X^{[\bullet]}=[\Omega_X^{[p]}\to \Omega_X^{[p+1]}\to \cdots\to \Omega_X^{[d]}]$, which lives in degrees $\{p, p+1,\dots, d\}$. 
\end{definition}
The twisted Hodge groups and numbers defined in \eqref{eqn:TwistedHodge} naturally generalize for V-manifolds.
\begin{definition}
	\label{def:TwistedHodgeGroups}
	Given a V-manifold $X$ and a holomorphic line bundle $\mathcal{L}$ on it,  for any $p,q$, the $(p,q)$-th \textit{twisted Hodge group} is the following cohomology using reflexive differentials:
	\begin{equation}
	H^{p,q}(X, \mathcal{L}):=H^q(X, \Omega^{[p]}_X\otimes \mathcal{L}),
	\end{equation}
	whose dimension is denote by $h^{p,q}(X, \mathcal{L})$, called the $(p,q)$-th \textit{twisted Hodge number}. 
\end{definition}

\subsection{Relation with constant Hodge module}
Let $X$ be a V-manifold of dimension $d$, then $X$ is a rational homology manifold, and 
\begin{equation}
\IC_X(\QQ)\cong\QQ_X[d].
\end{equation}
The IC Hodge module $\ICH_X(\QQ)=(\QQ_X[d], \ICH_X, F_\bullet)$ coincides with the constant Hodge module $\QQ_X^H[d]$. 

The relation \eqref{eqn:grDRomega} between the constant Hodge module and the differentials for complex manifold naturally extends to V-manifolds, upon replacing differentials by reflexive differentials:
\begin{prop}
	\label{prop:ReflexiveDiffVmanifold}
	Let $X$ be a V-manifold of dimension $d$. Let $\ICH_X$ be the underlying filtered $\D_X$-module of the IC (or constant) Hodge module of $X$. For any integer $p$, there is a canonical isomorphism of coherent $\mathscr{O}_X$-modules:
	\begin{equation}
	\Omega_X^{[p]}\cong \gr^F_{-p} \DR (\ICH_X)[p-d].
	\end{equation}  
\end{prop}

\begin{proof}
	In the algebraic setting, since $X$ is a V-manifold,  $\ICH_X=\QQ_X^H[d]$, and $\Omega_X^{[p]}\cong\underline{\Omega}_X^{p}$ by Du Bois' \Cref{thm:DuBoisComplexVManifolds}. Now the assertion follows from Saito's characterization of Du Bois differentials \eqref{eqn:grDRIC-2}.
	
	In the complex analytic setting, this can be deduced as follows. For any local orbifold chart $(V, G, U\xrightarrow{\cong} V/G)$ of $X$, denote $\varpi\colon V\to U$ the natural map. Consider the canonical morphism of Hodge modules
	\begin{equation}
	\label{eqn:OrbifoldChartMorphismHM}
	\QQ_U^H[d]\to \R\varpi_*\QQ^H_V[d].
	\end{equation}
	By Saito's decomposition theorem for Hodge modules, which holds since $\varpi$ is projective, \eqref{eqn:OrbifoldChartMorphismHM} admits a retraction with complement supported in a proper subvariety of $U$. Applying the functor $\gr^F_{-p}\circ \DR$ to the morphism of underlying $\D$-modules in \eqref{eqn:OrbifoldChartMorphismHM}, we get a canonical morphism admitting a retraction with complement having proper support:
	\begin{equation}
	\label{eqn:MorphismWithSection}
	\gr^F_{-p}(\DR(\ICH_U)) \to \gr^F_{-p}(\DR(\varpi_+\omega_V))\cong \R\varpi_*\gr^F_{-p}\DR(\omega_V)\cong \R\varpi_*\Omega^p_V[d-p]\cong \Omega_U^{[p]}[d-p],
	\end{equation}
	where the three isomorphisms use \Cref{thm:Strictness}, \eqref{eqn:grDRomega}, and \cite[Lemma 2.46]{Steenbrink-MHSonVanishingCohomology-1977} respectively.
	
	Now since the $\Omega_U^{[p]}[d-p]$ is torsion-free hence cannot have a non-zero direct summand with proper support, \eqref{eqn:MorphismWithSection} must be an isomorphism. Therefore all the \textit{canonical} isomorphisms \eqref{eqn:MorphismWithSection} from all orbifold charts glue into an isomorphism as claimed.
\end{proof}

\section{Semismall morphisms}

\subsection{Basic definitions and notations}
Recall that a proper surjective morphism  $\pi\colon X\to Y$  between irreducible varieties (or complex analytic spaces) is called \textit{semismall} if 
\begin{equation}
\dim (X\times_YX)\leq \dim (X);
\end{equation}
or equivalently, for any integer $k\geq 1$,
\begin{equation}
\operatorname{codim}\{y\in Y~|~ \dim f^{-1}(y)\geq k\}\geq 2k.
\end{equation}
Note that there exists a smooth dense open subset $U$ of $Y$ with complement of codimension at least 2, such that $\pi$ is finite over $U$. In particular, $\dim(X)=\dim(Y)$.

Let $\pi\colon X\to Y$ be a semismall morphism. By the Thom--Mather theory (\cite{GoreskyMacPherson-MorseBook}), one can stratify $Y$ by locally closed smooth subvarieties $$Y=\bigsqcup_{\alpha}Y_{\alpha},$$
such that $\pi_{\alpha}\colon X_{\alpha}\to Y_{\alpha}$ is a topological fiber bundle, where $X_\alpha=\pi^{-1}(Y_\alpha)$ and $\pi_{\alpha}=\pi|_{X_{\alpha}}$. The semismallness condition says that $\dim (X_{\alpha}\times_{Y_{\alpha}}X_{\alpha})\leq \dim X$, or equivalently, $\frac{1}{2}\operatorname{codim}(Y_\alpha)$ is at least the relative dimension of $\pi_{\alpha}$.  

A stratum $Y_\alpha$ is called \textit{relevant} if $\dim (X_{\alpha}\times_{Y_{\alpha}}X_{\alpha})=\dim X$. In this case, denote by 
\begin{equation}
c_{\alpha}:=\operatorname{codim}(Y_{\alpha}),    
\end{equation}
which is an even integer, then the dimension of fibers over $Y_{\alpha}$ is $\frac{1}{2}c_{\alpha}$. We have the following local system on $X_\alpha$:
\begin{equation}
\mathbb{L}_{\alpha}:=\R^{c_{\alpha}}\pi_{\alpha, *}\QQ_{X_{\alpha}}.
\end{equation}
Note that for dimension reason, the local system $\mathbb{L}_{\alpha}$ is completely determined by the permutation action of $\pi_1(Y_{\alpha})$ on the set of \textit{top-dimensional} irreducible components of the fibers over $Y_{\alpha}$, in particular, it is semisimple. The local system $\mathbb{L}_{\alpha}$ naturally underlies a variation of Hodge structure of weight $c_\alpha$ and of Tate type (i.e. the only nontrivial summand is in degree $(\frac{c_\alpha}{2}, \frac{c_\alpha}{2})$), which we denote by 
\begin{equation}
L_{\alpha}=(\mathbb{L}_\alpha, \mathcal{L}_{\alpha}, \nabla, F^\bullet).
\end{equation}

Denote  by $\overline{Y_{\alpha}}$ the closure of $Y_{\alpha}$ and
\begin{equation}
i_{\alpha}\colon \overline{Y_{\alpha}}\to Y
\end{equation}
the natural closed immersion.

\subsection{Decomposition theorems for semismall morphisms}

The decomposition theorem of Beilinson--Bernstein--Deligne--Gabber \cite{BBDG} says that the pushforward of the intersection complex via a proper morphism decomposes into a direct sum of shifts of perverse sheaves. In general, it can be difficult to determine the supports and local systems appearing in the result of decomposition theorem. However, in the semismall case, the decomposition theorem takes a much simpler and more precise form:

\begin{theorem}[Borho--MacPherson \cite{BorhoMacPherson}]
	Let $\pi\colon X\to Y$ be a proper semismall morphism between varieties of dimension $d$. The notations $Y_{\alpha}, \mathbb{L}_{\alpha}, i_{\alpha}$ are as above. Assume that $X$ is smooth\footnote{It suffices to assume rational smoothness, for example, varieties with quotient singularities.}. Then we have an isomorphism of perverse sheaves
	\begin{equation}
	\label{eqn:BorhoMacPerson}
	\R\pi_*\QQ_X[d]\cong \bigoplus_{\alpha \text{ relevant }}i_{\alpha,*}\IC_{\overline{Y_{\alpha}}}(\mathbb{L}_{\alpha}),
	\end{equation}
	where $\IC_{\overline{Y_{\alpha}}}(\mathbb{L}_{\alpha})=j_{\alpha, !*}\mathbb{L}_{\alpha}[d-c_\alpha]$ is the intermediate extension of the perverse sheaf $\mathbb{L}_{\alpha}[d-c_\alpha]$ on $Y_{\alpha}$ to $\overline{Y_{\alpha}}$,  where $j_\alpha\colon Y_{\alpha}\hookrightarrow \overline{Y_{\alpha}}$ is the open immersion.
\end{theorem}

Saito lifts the decomposition theorem to the level of Hodge modules \cite{Saito90-MHM}. In the semismall case, it takes the following more precise form, enriching both sides of \eqref{eqn:BorhoMacPerson} with the structure of pure Hodge modules.

%Recall that for a smooth algebraic variety $X$, there is a natural weight-$d$ pure polarizable Hodge module over $X$, namely, $\QQ_X^H[d]=(\QQ_X[d], \omega_X, F_\bullet \omega_X)$.  

\begin{theorem}
	\label{thm:BorhoMacPherson-HM}
	Let $\pi\colon X\to Y$ be a proper semismall morphism between varieties of dimension $d$. The notations $Y_{\alpha}, L_{\alpha}, i_{\alpha}$ are as above. Assume that $X$ is smooth. Then we have an isomorphism of Hodge modules
	\begin{equation}
	\label{eqn:SaitoSemiSmall}
	\R\pi_*\QQ^H_X[d]\cong \bigoplus_{\alpha \text{ relevant }}i_{\alpha,*}\ICH_{\overline{Y_{\alpha}}}(L_{\alpha}),
	\end{equation}
	where $\ICH_{\overline{Y_{\alpha}}}(L_{\alpha})$ is the unique pure Hodge module on $Y$ with strict support $\overline{Y_{\alpha}}$ that restricts to the (shifted) variation of Hodge structure $L_{\alpha}$ on $Y_{\alpha}$.
\end{theorem}

\begin{proof}
	G\"ottsche--Soergel \cite[Theorem 5]{GottscheSoergel} proved the special case where $L_{\alpha}$ has trivial local system, but their proof can be easily adapted to show the general statement. Since we cannot find this precise statement in the literature, let us include a proof here for the convenience of the reader.
	
	By Saito's theory of mixed Hodge modules, $\R\pi_*\QQ^H_X[d]$ is an (a priori mixed) Hodge module and the strict support decomposition of cohomologies of $\R\pi_*\QQ^H_X[d]$ must be the one given by the decomposition  \eqref{eqn:BorhoMacPerson} at the level of perverse sheaves in Borho--MacPherson's \Cref{thm:BorhoMacPherson-HM}. Therefore, we have an isomorphism of pure Hodge modules
	\begin{equation}
	\label{eqn:DecompositionWithUnknownVHS}
	\R\pi_*\QQ^H_X[d]\cong \bigoplus_{\alpha \text{ relevant }}i_{\alpha,*}\ICH_{\overline{Y_{\alpha}}}(L'_{\alpha}),
	\end{equation}
	for some variation of Hodge structure $L'_{\alpha}$ on some dense open subvariety of $\overline{Y_{\alpha}}$ with the underlying local system $\mathbb{L}_\alpha$ same as $L_{\alpha}$.
	
	For a relevant stratum indexed by $\alpha$, to determine the variation of Hodge structure $L'_{\alpha}$, we apply the functor $\mathcal{H}^0\circ j_\alpha^* \circ i_\alpha^*$ to \eqref{eqn:DecompositionWithUnknownVHS} and obtain that
	\begin{equation}
	\mathcal{H}^0 j_\alpha^*  i_\alpha^*(\R\pi_*\QQ^H_X[d])\cong L^{'H}_\alpha[d_\alpha].
	\end{equation}
	By base change, $j_\alpha^*  i_\alpha^*(\R\pi_*\QQ^H_X[d])\cong \R\pi_{\alpha, *}\QQ^H_{X_\alpha}[d]$. Since $\pi_\alpha\colon X_\alpha\to Y_\alpha$ is a topological fibration, we have an isomorphism of variations of Hodge structures:
	\begin{equation}
	L'_\alpha\cong \R^{c_\alpha}\pi_{\alpha, *}\QQ_{X_{\alpha}}=L_\alpha.
	\end{equation}

	%    Alternatively, we can use the motivic decomposition established by de Cataldo and Migliorini \cite{deCataldoMigliorini-semismallMotive}. More precisely, by \cite[Theorem 2.3.8]{deCataldoMigliorini-semismallMotive}, we have an isomorphism in $\CHM(Y)$,  the category of relative Chow motives constructed by Corti--Hanamura \cite{CortiHanamura} (generalizing \cite{DeningerMurre}):
	%
	%   \begin{equation}
	%        (X\xrightarrow{\pi} Y)\cong  \bigoplus_{\alpha \text{ relevant }}(X\xrightarrow{\pi} Y, P_{\alpha})
	%   \end{equation}
	%   where $P_{\alpha}\in \CH_d(X\times_Y X)\cong \End(R\pi_* \QQ_X[d])$ (the isomorphism is for dimension reason by the semismallness assumption, see \cite[Lemma 2.3.7]{deCataldoMigliorini-semismallMotive}) is the projector corresponding to the relative projector onto $i_{\alpha,*}\IC_{\overline{Y_{\alpha}}}(L_{\alpha})$ in the direct-sum decomposition \eqref{eqn:BorhoMacPerson}.
	%   Via the above identification, the algebraic cycle $P_{\alpha}$ is constructed in \cite{deCataldoMigliorini-semismallMotive}. More precisely, it is  [TO DO].

\end{proof}

\section{Proof of the main results}
\label{sec:Proof}

We prove  \Cref{thm:main-cohomology}, \Cref{thm:main-numbers}, \Cref{cor:Frolicher} and \Cref{cor:BFK-HH} in this section. Let $S$ be a compact complex surface. Set $X=\Hilb^n(S)$ and $Y=S^{(n)}$.  

\subsection{Stratification and semismallness}
The Hilbert--Chow morphism 
\begin{equation}
\pi\colon  X\to Y
\end{equation}
is semismall. There is a natural stratification of $\pi$ by the types of supports of subschemes. More precisely, for a partition $\lambda$ of $n$, we always write 
\begin{equation}
\lambda=(\lambda_1\geq \lambda_2\cdots \geq \lambda_\ell)=(1^{a_1}2^{a_2}\cdots r^{a_r}),
\end{equation}
where for any $k\geq 1$,  $a_k\geq 0$ is the number of $k$'s appearing in the partition $(\lambda_1\geq \lambda_2\cdots \geq \lambda_\ell)$. In particular,  
\begin{equation}\label{eqn:SumOfkak=n}
\sum_{k=1}^r ka_k=n
\end{equation}
and the \textit{length} of $\lambda$ is
\begin{equation}
\label{eqn:SumOfak=l}
|\lambda|=\ell=\sum_{k=1}^r a_k.
\end{equation}

The natural stratification of $\pi$ is indexed by the partitions of $n$ as follows: given a partition $\lambda\dashv n$ as above, define the locally closed subvariety
\begin{equation}
Y_{\lambda}:=S_{\lambda}^{(n)}:=\left\{\sum_{i=1}^{\ell} \lambda_i x_i ~|~ x_i\in S \text{ are distinct}\right\}.
\end{equation}
Then $Y=\bigsqcup_{\lambda\dashv n} Y_{\lambda}$ is a stratification by locally closed subvarieties with smooth strata. Note that $\dim(Y_{\lambda})=2|\lambda|$ and  its codimension in $Y$ is
\begin{equation}
c_{\lambda}=2(n-|\lambda|).
\end{equation}

Let $X_{\lambda}:=\pi^{-1}(Y_{\lambda})$ and $\pi_{\lambda}:=\pi|_{X_{\lambda}}$. The morphism $\pi_\lambda\colon X_{\lambda}\to Y_{\lambda}$ is an isotrivial fibration with all fibers isomorphic to 
\begin{equation}
F_{\lambda}:=\prod_{i=1}^{\ell}\mathbb{B}_{\lambda_i},
\end{equation}
where $\mathbb{B}_m=\Hilb^m(\CC^2)_0$ is the $m$-th \textit{Brian\c{c}on variety} parametrizing subschemes of $\CC^2$ of length $m$ supported at the origin, which is an \textit{irreducible} variety of dimension $m-1$ by  \cite{Briancon}. Therefore, the fiber dimension of $\pi_\lambda$ is 
\begin{equation}
\dim(F_\lambda)=\sum_{i=1}^{\ell}(\lambda_i-1)=n-|\lambda|,
\end{equation} 
which is exactly half of the codimension of $Y_\lambda$.  Hence \textit{all strata are relevant}. 

For each $\lambda\dashv n$,  since $F_{\lambda}$ is irreducible, the local system $\mathbb{L}_{\lambda}$ is the trivial one, hence the variation of Hodge structure $L_\lambda$ has to be the Tate object $\QQ(-\frac{c_\lambda}{2})=\QQ(|\lambda|-n)$ of weight $c_\lambda$.

\subsection{Normalization of strata}
Now we study each stratum. Given a partition $\lambda=(1^{a_1}2^{a_2}\cdots r^{a_r})$ of $n$, the closure of the stratum $Y_\lambda$ in $S^{(n)}$, denoted by  $\overline{Y_{\lambda}}$,  is in general not normal. Its normalization admits the following natural description:

\begin{align*}
\nu_\lambda\colon S^{(a_1)}\times \cdots\times S^{(a_r)}&\to \overline{Y_{\lambda}}\\
(z_1, \dots, z_r)&\mapsto \sum_{k=1}^{r}kz_k,
\end{align*}
where we identify a point in $S^{(a)}$ with an effective zero-cycle of length $a$ on $S$. The morphism $\nu_\lambda$ is an isomorphism over $Y_\lambda$, hence is a \textit{birational} morphism. Moreover, by construction, $\nu_\lambda$ is finite. We denote the closed immersion 
\begin{equation}
i_\lambda\colon \overline{Y_{\lambda}} \to S^{(n)},
\end{equation}
and the composition 
\begin{equation}
\iota_\lambda:=i_{\lambda}\circ \nu_{\lambda}\colon S^{(a_1)}\times \cdots\times S^{(a_r)}\to S^{(n)}.
\end{equation}
\begin{lemma}
	\label{lemma:PullBackLineBundle}
	For any line bundle $L$ on $S$, let $L_{(n)}$ be the induced natural line bundle on $S^{(n)}$ defined as in \eqref{eqn:L(n)}. Then for any partition $\lambda=(1^{a_1}2^{a_2}\cdots r^{a_r})$ of $n$, we have an isomorphism:
	\begin{equation}
	\label{eqn:PullBackLineBundle}
	\iota_\lambda^*L_{(n)}\cong L_{(a_1)}\boxtimes L^{\otimes 2}_{(a_2)}\boxtimes\cdots\boxtimes L^{\otimes r}_{(a_r)}.
	\end{equation}
\end{lemma}
\begin{proof}
	Consider the commutative diagram:

	\begin{equation}
	\begin{tikzcd}
	S^{a_1}\times \cdots\times S^{a_r}\arrow[r, "\widetilde{\iota_{\lambda}}"] \arrow[d, "\varpi'"'] & S^n \arrow[d, "\varpi"] \\
	S^{(a_1)}\times \cdots\times S^{(a_r)} \arrow[r, "\iota_{\lambda}"'] & S^{(n)} 
	\end{tikzcd}
	\end{equation}
	where the top arrow sends $(x_{k,j}; {1\leq k\leq r}, 1\leq j\leq a_k)$ to the sequence with $x_{k,j}$ repeated $k$  times for each $k, j$.
	
	By construction,   $\varpi^*L_{(n)}=L^{\boxtimes n}$, hence 
	\begin{equation}
	\varpi'^{*}\iota_\lambda^*L_{(n)}\cong \widetilde{\iota_{\lambda}}^*(L^{\boxtimes n})\cong L^{\boxtimes a_1}\boxtimes (L^{\otimes 2})^{\boxtimes a_2}\boxtimes\cdots\boxtimes (L^{\otimes r})^{\boxtimes a_r}.
	\end{equation}
	Therefore, by \eqref{eqn:V-bundle},
	\begin{equation}
	\iota_\lambda^*L_{(n)}\cong \varpi'_*(L^{\boxtimes a_1}\boxtimes (L^{\otimes 2})^{\boxtimes a_2}\boxtimes\cdots\boxtimes (L^{\otimes r})^{\boxtimes a_r})^{\mathfrak{S}_{a_1}\times \cdots\times \mathfrak{S}_{a_r}}\cong L_{(a_1)}\boxtimes L^{\otimes 2}_{(a_2)}\boxtimes\cdots\boxtimes L^{\otimes r}_{(a_r)},
	\end{equation}
	as is claimed.
	%Alternatively, one can show the lemma by looking directly at fibers. Given effective zero-cycles $z_k\in S^{(a_k)}$ for $k=1, \dots, r$, write $z_k=x_{k,1}+\cdots+x_{k, a_k}$ for each $k$. Then the map $\iota_\lambda\colon S^{(a_1)}\times \cdots\times S^{(a_r)}\to S^{(n)}$ sends $(z_1, \cdots, z_r)$ to the zero-cycle
	%$$\sum_{k=1}^rz_k=\sum_{k=1}^r\sum_{j=1}^{a_k}kx_{k,j}.$$
	%It is easy to see that the fibers of both sides of \cref{eqn:PullBackLineBundle} at the point $(z_1, \cdots, z_r)$ are canonically isomorphic to  
	%\begin{equation}
	%    \bigotimes_{k=1}^r\bigotimes_{j=1}^{a_k}L^{\otimes k}_{x_{k,j}}.
	%\end{equation}
\end{proof}
\begin{rmk}
	\Cref{lemma:PullBackLineBundle} does not generalize to vector bundles of higher rank. The reason is that for a vector bundle $E$ on $S$ of rank $r>1$, the coherent sheaf $E_{(n)}:=\varpi_*(E^{\boxtimes n})^{\mathfrak{S}_n}$ is not locally free but only reflexive (in fact a V-bundle). In general $\iota_{\lambda}^*E_{(n)}$ is not a tensor product of sheaves pulling back from the factors.
\end{rmk}

\subsection{Twisted Hodge groups of symmetric powers}
As a preparation towards the computation of twisted Hodge groups/numbers of the pair $(\Hilb^nS, L_n)$, we first need to compute the twisted Hodge groups/numbers of the pair $(S^{(n)}, L_{(n)})$. Note that $S^{(n)}$ is not smooth but is a V-manifold. The twisted Hodge groups are defined using reflexive differentials; see \Cref{def:TwistedHodgeGroups}.
\begin{prop}
	\label{prop:SymmetricPowers}
	Let $S$ be a compact complex surface and let $L$ be a holomorphic line bundle on $S$. Then for any integer $n\geq 0$, we have a canonical isomorphism of bigraded vector spaces:
	\begin{equation}
	\label{eqn:TwistedHodgeSymmetricPower}
	\bigoplus_{p,q\geq 0} H^{p,q}(S^{(n)}, L_{(n)})x^py^q\cong \Sym^n\left(\bigoplus_{i,j\geq 0}H^{i,j}(S, L)x^iy^j\right).
	\end{equation}
	Here  $\Sym^n$ is taken in the super sense with respect to the grading given by the total degree of $x$ and $y$.
	
	More succinctly, 
	\begin{equation}
	H^{\#,\star}(S^{(n)}, L_{(n)}) \cong  \Sym^n\left(H^{\#,\star}(S, L)\right).
	\end{equation}
\end{prop}
\begin{proof}
	Denote by $\varpi\colon S^{n}\to S^{(n)}$ the natural quotient map. Then $\Omega^{[p]}_{S^{(n)}}\cong(\varpi_*\Omega^p_{S^n})^{\mathfrak{S}_n}$  is a V-bundle on $S^{(n)}$; see for example \cite[Lemma 2.46]{PetersSteenbrink-MHSBook}. Therefore, $(\varpi^{*}\Omega^{[p]}_{S^{(n)}})^{\vee \vee}\cong \Omega_{S^n}^p$; see \eqref{eqn:V-bundle}. Since $\varpi^*L_{(n)}\cong L^{\boxtimes n}$, we have 
	\begin{equation}
	\varpi^*(\Omega^{[p]}_{S^{(n)}}\otimes L_{(n)})^{\vee \vee}\cong \Omega_{S^n}^p\otimes L^{\boxtimes n}.
	\end{equation}
	This implies that (see \eqref{eqn:V-bundle})
	\begin{equation}
	\Omega^{[p]}_{S^{(n)}}\otimes L_{(n)}\cong \varpi_*(\Omega_{S^n}^p\otimes L^{\boxtimes n})^{\mathfrak{S}_n}.
	\end{equation}
	Now we can compute the cohomology group that we are interested in:
	\begin{align*}
	& H^{q}(S^{(n)}, \Omega^{[p]}_{S^{(n)}}\otimes L_{(n)})\\
	\cong\quad & H^q(S^n, \Omega_{S^n}^p\otimes L^{\boxtimes n})^{\mathfrak{S}_n}\\
	\cong \quad &\left(\bigoplus_{i_1+\cdots i_n=p} H^q(S^n, \boxtimes_{k=1}^n(\Omega^{i_k}_S\otimes L))\right)^{\mathfrak{S}_n}\\
	\cong\quad  &\left(\bigoplus_{\substack{i_1+\cdots i_n=p\\j_1+\cdots j_n=q}} \bigotimes_{k=1}^n H^{j_k}(S, \Omega^{i_k}_S\otimes L)\right)^{\mathfrak{S}_n},
	\end{align*}
	where the first isomorphism uses the exactness of the functor $-^{\mathfrak{S}_n}$, the second isomorphism uses $\Omega^p_{S^n}\cong \bigwedge^p(p_1^*\Omega^1_S\oplus \cdots\oplus p_n^*\Omega^1_S)$, and the last isomorphism is the K\"unneth formula.
	Taking the direct sum over all $p, q\geq 0$, this implies that 
	\begin{align*}
	&\bigoplus_{p,q\geq 0} H^{p,q}(S^{(n)}, L_{(n)})x^py^q\\
	\cong \quad&
	\left(\bigoplus_{\substack{i_1,\cdots,i_n\\j_1,\cdots,j_n}} \bigotimes_{k=1}^n H^{i_k, j_k}(S, L)x^{i_k}y^{j_k}\right)^{\mathfrak{S}_n}\\
	\cong \quad&\left(\bigotimes_{k=1}^n \left(\bigoplus_{i,j\geq 0}H^{i,j}(S, L)x^iy^j\right)\right)^{\mathfrak{S}_n}\\
	\cong \quad&\Sym^n\left(\bigoplus_{i,j\geq 0}H^{i,j}(S, L)x^iy^j\right).
	\end{align*}
	This is exactly \eqref{eqn:TwistedHodgeSymmetricPower}.
\end{proof}

\subsection{Isomorphism of Hodge modules and of $\mathscr{D}$-modules}

Let the notation be as above: $X=\Hilb^nS$, $Y=S^{(n)}$.

\begin{prop}
	\label{prop:Isom-HodgeModules}
	We have an isomorphism of pure polarizable Hodge modules of weight $2n$ on $Y$:
	\begin{equation}
	\label{eqn:IsomHodgeModule}
	\R\pi_*\QQ^H_X[2n]\cong \bigoplus_{\substack{\lambda\dashv n\\\lambda=(1^{a_1}\dots r^{a_{r}})}} \iota_{\lambda,*} \ICH_{S^{(a_1)}}\boxtimes \cdots \boxtimes \ICH_{S^{(a_r)}}(\QQ(|\lambda|-n)).
	\end{equation}
\end{prop}

\begin{proof}
	This is \cite[Theorem 4]{GottscheSoergel}, we include the proof for the convenience of the reader.   Applying \Cref{thm:BorhoMacPherson-HM} to the semismall morphism $\pi\colon X\to Y$, we obtain an isomorphism in $\HM^\p(Y,2n)$:
	\begin{equation}
	\label{eqn:IsomHM-1}
	\R\pi_*\QQ^H_X[2n]\cong \bigoplus_{\lambda\dashv n }i_{\lambda,*}\ICH_{\overline{Y_{\lambda}}}(\QQ(|\lambda|-n)),
	\end{equation}
	where $\ICH_{\overline{Y_{\alpha}}}(\QQ(|\lambda|-n))$ is the unique pure polarizable Hodge module with strict support $\overline{Y_{\alpha}}$ that restricts to the weight $c_\lambda$ variation of Hodge structure $\QQ(|\lambda|-n)$ on $Y_{\alpha}$.
	
	For each partition $\lambda$ of $n$, write $\lambda=(1^{a_1}\dots r^{a_{r}})$ as before,  since the normalization $\nu_\lambda\colon S^{(a_1)}\times \cdots\times S^{(a_r)}\to \overline{Y_{\lambda}}$  is finite and birational, by \Cref{lemma:BirationalFinite}, 
	\begin{equation}
	\label{eqn:IsomHM-2}
	\nu_{\lambda,*}\ICH_{S^{(a_1)}\times \cdots \times S^{(a_r)}}(\QQ(|\lambda|-n))\cong \ICH_{\overline{Y_{\lambda}}}(\QQ(|\lambda|-n)).
	\end{equation}
	Since $\iota_\lambda:=i_{\lambda}\circ \nu_{\lambda}$, combining \eqref{eqn:IsomHM-1} and \eqref{eqn:IsomHM-2}, and using $\ICH_{S^{(a_1)}\times \cdots \times S^{(a_r)}}\cong \ICH_{S^{(a_1)}}\boxtimes \cdots \boxtimes \ICH_{S^{(a_r)}}$,  we obtain \eqref{eqn:IsomHodgeModule}.
\end{proof}

Taking the underlying filtered $\D$-modules on both sides of \eqref{eqn:IsomHodgeModule} in \Cref{prop:Isom-HodgeModules}, we get the following result.

\begin{cor}
	\label{cor:IsomDModules}
	Notation is as above. Let $\ICH_{S^{(a)}}$ denote the underlying right $\D$-module of the Hodge module $\ICH_{S^{(a)}}(\QQ)$. We have an isomorphism of filtered right $\D_Y$-modules:
	\begin{equation}
	\label{eqn:IsomDModules}
	\pi_+(\omega_X, F_\bullet)\cong  \bigoplus_{\substack{\lambda\dashv n\\\lambda=(1^{a_1}\dots r^{a_{r}})}} \iota_{\lambda,+}( \ICH_{S^{(a_1)}}\boxtimes \cdots \boxtimes \ICH_{S^{(a_r)}}, F_{\bullet-|\lambda|+n}),
	\end{equation}
	where the filtration on the left-hand side is the usual one (see \eqref{eqn:TrivialHM}), and the filtration on the right-hand side is the tensor product of the usual ones shifted by $|\lambda|-n$.
\end{cor}

\subsection{Putting everything together}

%Given a compact complex surface $S$, for a positive integer $a$, the symmetric power $S^{(a)}$ is a variety with quotient singularities (V-manifold). Recall that the twisted Hodge groups for V-manifold is defined using reflexive differentials: $H^{i, j}(S^{(a)},  L_{(a)}):=H^{j}(S^{(a)}, \Omega_{S^{(a)}}^{[i]}\otimes L_{(a)})$.
The following proposition contains the key computation of the proof.

\begin{prop}
	\label{prop:PushforwardOfOmega}
	We have a canonical isomorphism in $\operatorname{D^b_{coh}}(S^{(n)})$:
	\begin{equation}
	\label{eqn:RpiOmegap}
	\R\pi_*\Omega_{\Hilb^nS}^p\cong \bigoplus_{\substack{\lambda\dashv n\\\lambda=(1^{a_1}\dots r^{a_{r}})}}\bigoplus_{\sum i_k=p+|\lambda|-n}\iota_{\lambda,*}\left(\Omega_{S^{(a_1)}}^{[i_1]}\boxtimes\cdots\boxtimes\Omega_{S^{(a_r)}}^{[i_r]}[|\lambda|-n]\right).
	\end{equation}
\end{prop}
\begin{proof}
	We apply the functor $\gr^F_{-p}\circ \DR$ to both sides of \eqref{eqn:IsomDModules} in \Cref{cor:IsomDModules}.    
	On the left-hand side, by Saito's \Cref{thm:Strictness} on strictness of direct images (and \Cref{rmk:SingularTarget}), together with \eqref{eqn:grDRomega}, we have
	\begin{equation}
	\gr^F_{-p}\DR(\pi_+\omega_X)\cong \R\pi_*\gr^F_{-p}\DR(\omega_X)\cong \R\pi_*(\Omega_X^p[2n-p]).
	\end{equation}
	Similarly, for the right-hand side, for any partition $\lambda=(1^{a_1}2^{a_2}\cdots r^{a_r})$ of $n$, we have isomorphisms
	\begin{align*}
	& \gr^F_{-p}\DR (\iota_{\lambda,+}( \ICH_{S^{(a_1)}}\boxtimes \cdots \boxtimes \ICH_{S^{(a_r)}}, F_{\bullet-|\lambda|+n}))\\
	\cong \quad&\gr^F_{-p-|\lambda|+n}\DR (\iota_{\lambda,+}( \ICH_{S^{(a_1)}}\boxtimes \cdots \boxtimes \ICH_{S^{(a_r)}}))\\
	\cong\quad &\iota_{\lambda,*}(\gr^F_{-p-|\lambda|+n}\DR ( \ICH_{S^{(a_1)}}\boxtimes \cdots \boxtimes \ICH_{S^{(a_r)}}))\\
	\cong \quad&\iota_{\lambda,*}\left(\bigoplus_{\sum i_k=p+|\lambda|-n}\gr^F_{-i_1}\DR ( \ICH_{S^{(a_1)}})\boxtimes \cdots \boxtimes \gr^F_{-i_r}\DR ( \ICH_{S^{(a_r)}})\right)\\
	\cong\quad & \iota_{\lambda,*}\left(\bigoplus_{\sum i_k=p+|\lambda|-n}\Omega_{S^{(a_1)}}^{[i_1]}[2a_1-i_1]\boxtimes\cdots\boxtimes\Omega_{S^{(a_r)}}^{[i_r]}[2a_r-i_r]\right)\\
	\cong \quad& \bigoplus_{\sum i_k=p+|\lambda|-n}\iota_{\lambda,*}\left(\Omega_{S^{(a_1)}}^{[i_1]}\boxtimes\cdots\boxtimes\Omega_{S^{(a_r)}}^{[i_r]}[|\lambda|-p+n]\right)
	\end{align*}
	where the second isomorphism uses Saito's \Cref{thm:Strictness} (and \Cref{rmk:SingularTarget}), the fourth isomorphism follows from \Cref{prop:ReflexiveDiffVmanifold}.
	
	According to the above computations, applying the functor $\gr^F_{-p}\circ \DR$ to both sides of \eqref{eqn:IsomDModules} yields the claimed isomorphism.
\end{proof}

\begin{prop}
	\label{prop:TwistedHodgeGroupsOfHilb}
	For any nonnegative integers $p, q, n$, we have a canonical isomorphism
	\begin{equation}
	H^{p,q}(\Hilb^nS, L_n)\cong \bigoplus_{\substack{\lambda\dashv n\\\lambda=(1^{a_1}\dots r^{a_{r}})}} \bigoplus_{\substack{\sum i_k=p+|\lambda|-n\\\sum j_k=q+|\lambda|-n}}\bigotimes_{k=1}^r H^{i_k, j_k}(S^{(a_k)},  L^{\otimes k}_{(a_k)}).
	\end{equation}
\end{prop}
\begin{proof}

	Still denote $X=\Hilb^nS$.  We have the following chain of isomorphisms:
	\begin{align}
	\begin{split}
	& H^q(X, \Omega^p_{X}\otimes L_n)\\
	\cong \quad& H^q(Y, \R\pi_*\Omega_X^p\otimes L_{(n)})\\
	\cong \quad& \bigoplus_{\lambda\dashv n}\bigoplus_{\sum i_k=p+|\lambda|-n}H^{q+|\lambda|-n}(S^{(a_1)}\times \cdots\times S^{(a_r)}, \Omega_{S^{(a_1)}}^{[i_1]}\boxtimes\cdots\boxtimes\Omega_{S^{(a_r)}}^{[i_r]}\otimes \iota_{\lambda}^*L_{(n)})\\
	\cong \quad& \bigoplus_{\lambda\dashv n}\bigoplus_{\sum i_k=p+|\lambda|-n}H^{q+|\lambda|-n}(S^{(a_1)}\times \cdots\times S^{(a_r)}, (\Omega_{S^{(a_1)}}^{[i_1]}\boxtimes\cdots\boxtimes\Omega_{S^{(a_r)}}^{[i_r]})\otimes (L_{(a_1)}\boxtimes \cdots \boxtimes L^r_{(a_r)})).\\
	\cong\quad& \bigoplus_{\lambda\dashv n}\bigoplus_{\sum i_k=p+|\lambda|-n} H^{q+|\lambda|-n}\left(\prod_{k=1}^rS^{(a_k)}, \boxtimes_{k=1}^r (\Omega_{S^{(a_k)}}^{[i_k]}\otimes L^k_{(a_k)})\right).\\
	\cong \quad&\bigoplus_{\lambda\dashv n}\bigoplus_{\substack{\sum i_k=p+|\lambda|-n\\\sum j_k=q+|\lambda|-n}}\bigotimes_{k=1}^r H^{j_k}(S^{(a_k)}, \Omega^{[i_k]}_{S^{(a_k)}}\otimes L^k_{(a_k)}),
	\end{split}
	\end{align}
	where the first isomorphism uses the definition $L_n=\pi^*L_{(n)}$ and the projection formula, the second isomorphism uses \eqref{eqn:RpiOmegap} in \Cref{prop:PushforwardOfOmega}, the third isomorphism follows from \Cref{lemma:PullBackLineBundle}, and the last isomorphism is by the K\"unneth formula.
\end{proof}

Now we can conclude:

\begin{proof}[Proof of \Cref{thm:main-cohomology}]
	By \Cref{prop:TwistedHodgeGroupsOfHilb}, the left-hand side of  \eqref{eqn:main-cohomology} can be computed as follows:
	\begin{align*}
	&\bigoplus_{n\geq 0}\bigoplus_{p,q\geq 0} H^{p,q}(\Hilb^nS, L_n)x^py^qt^n\\
	\cong\quad& \bigoplus_{n\geq 0}\bigoplus_{p,q\geq 0}\left(\bigoplus_{\substack{\lambda\dashv n\\\lambda=(1^{a_1}\dots r^{a_{r}})}} \bigoplus_{\substack{\sum i_k=p+|\lambda|-n\\\sum j_k=q+|\lambda|-n}}\bigotimes_{k=1}^r H^{i_k, j_k}(S^{(a_k)},  L^k_{(a_k)})\right)x^py^qt^n.\\
	\cong\quad& \bigoplus_{r\geq 0}\bigoplus_{a_1, \dots, a_r\geq 0}\bigoplus_{\substack{i_1, \dots,i_r\\j_1,\dots, j_r}}\left(\bigotimes_{k=1}^r H^{i_k, j_k}(S^{(a_k)},  L^k_{(a_k)})\right)x^{\sum i_k+\sum(k-1)a_k}y^{\sum j_k+\sum (k-1)a_k}t^{\sum ka_k}.\\
	\cong\quad& \bigoplus_{r\geq 0}\bigoplus_{a_1, \dots, a_r\geq 0}\bigoplus_{\substack{i_1, \dots,i_r\\j_1,\dots, j_r}}\bigotimes_{k=1}^r \left(H^{i_k, j_k}(S^{(a_k)},  L^k_{(a_k)})x^{i_k}y^{j_k}(x^{k-1}y^{k-1}t^k)^{a_k}\right).\\
	\cong\quad& \bigoplus_{r\geq 0}\bigoplus_{a_1, \dots, a_r\geq 0}\bigotimes_{k=1}^r \left(\bigoplus_{i_k, j_k}H^{i_k, j_k}(S^{(a_k)},  L^k_{(a_k)})x^{i_k}y^{j_k}(x^{k-1}y^{k-1}t^k)^{a_k}\right).\\
	\end{align*}
	where the first isomorphism is by \Cref{prop:TwistedHodgeGroupsOfHilb}, the second isomorphism is obtained by a change of summation order and by noting the following numerical identities: $|\lambda|=\sum_k a_k$, $n=\sum_k ka_k$, and $p=(\sum_k i_k)+n-|\lambda|=\sum_k (i_k+(k-1)a_k)$, $q=(\sum_k j_k) +n-|\lambda|=\sum_k (j_k+(k-1)a_k)$.
	
	The right-hand side of \eqref{eqn:main-cohomology} can be computed using \Cref{prop:SymmetricPowers}:
	\begin{align*}
	&\Sym^\bullet \left(\bigoplus_{k\geq 1}\bigoplus_{p,q\geq 0}H^{p,q}(S, L^k)x^{p+k-1}y^{p+k-1}t^k\right)\\
	\cong\quad & \bigoplus_{r\geq 0}\bigoplus_{a_1, \dots, a_r\geq 0}\bigotimes_{k=1}^r \Sym^{a_k}\left(\bigoplus_{p,q}H^{p,q}(S, L^k)x^{p+k-1}y^{q+k-1}t^k\right)\\
	\cong \quad& \bigoplus_{r\geq 0}\bigoplus_{a_1, \dots, a_r\geq 0}\bigotimes_{k=1}^r \left(\Sym^{a_k}\left(\bigoplus_{p,q}H^{p,q}(S, L^k)x^{p}y^{q}\right)\cdot(x^{k-1}y^{k-1}t^k)^{a_k}\right)\\
	\cong\quad & \bigoplus_{r\geq 0}\bigoplus_{a_1, \dots, a_r\geq 0}\bigotimes_{k=1}^r \left(\bigoplus_{p,q}H^{p,q}(S^{(a_k)}, L_{(a_k)}^k)x^{p}y^{q}\cdot(x^{k-1}y^{k-1}t^k)^{a_k}\right)\\
	\end{align*}
	where \Cref{prop:SymmetricPowers} is used in the last isomorphism.
	
	Comparing the two end results above, we conclude that \eqref{eqn:main-cohomology} holds.
\end{proof}

\begin{proof}[Proof of \Cref{thm:main-numbers}]
	This follows from \Cref{thm:main-cohomology} by applying \cite[Lemma 3.3]{BFK}.
\end{proof}

\begin{proof}[Proof of \Cref{cor:Frolicher}]
	Given any compact complex surface $S$, G\"ottsche's formula \eqref{eqn:BettiNumbers} for Betti numbers holds by de Cataldo--Migliorini \cite[Theorem 5.2.1]{deCataldoMigliorini-DouadySpace-2000}, and G\"ottsche--Soergel's formula for Hodge numbers holds by   \Cref{cor:HodgeNumbers}. Therefore for any integer $i$, we have
	\begin{equation}
	\operatorname{b}_i(\Hilb^nS)=\sum_{p+q=i} h^{p,q}(\Hilb^nS).
	\end{equation}
	This implies that the Fr\"olicher spectral sequence dengerates at $E_1$-page.
\end{proof}
\begin{proof}[Proof of \Cref{cor:BFK-HH}]
	By Schuhmacher \cite{Schuhmacher-2004}, the Hochschild--Kostant--Rosenberg isomorphism holds for any complex manifold. Therefore for any compact complex surface and any holomorphic line bundle $L$ on $S$, we have an isomorphism for any $i, n$:
	\begin{equation}
	HH_i(\Hilb^nS, L_n)\cong \bigoplus_{q-p=i} H^{p,q}(\Hilb^nS, L_n).
	\end{equation}
	Consequently, 
	\begin{equation}
	\begin{split}
	&\bigoplus_{n\geq 0}\bigoplus_{i} HH_i(\Hilb^nS, L_n)y^it^n\\
	\cong\quad&\bigoplus_{n\geq 0}\bigoplus_{p,q} H^{p,q}(\Hilb^nS, L_n)y^{q-p}t^n\\
	\cong\quad& \Sym^\bullet \left(\bigoplus_{k\geq 1}\bigoplus_{i}\bigoplus_{q-p=i}H^{p,q}(S, L^{\otimes k})y^it^k\right)\\
	\cong \quad& \Sym^\bullet \left(\bigoplus_{k\geq 1}\bigoplus_{i}HH_i(S, L^{\otimes k})y^it^k\right),
	\end{split}
	\end{equation}
	where the second isomorphism follows from \eqref{eqn:main-cohomology} by setting $x=y^{-1}$. Now \eqref{eqn:HHCoeff} is proved for any compact complex surface $S$.
	
	To deduce \eqref{eqn:HH} , \eqref{eqn:HH*} and  \eqref{eqn:HS} for any compact complex surface $S$, it suffices to  further specialize to $L=\mathscr{O}_S$, $\omega_S^\vee$ and $\omega_S^{k-1}$  respectively in \eqref{eqn:HHCoeff}, since $(\omega_S)_n\cong \omega_{\Hilb^nS}$ as well as their tensor powers.
\end{proof}

\section{Deformation theory of Hilbert schemes of points on surfaces}
\label{sec:Deformation}

Let us first mention some previously known results on the deformation theory of $\Hilb^nS$.   
\begin{itemize}
	\item Fantechi \cite[Theorems 0.1 and 0.3]{Fantechi-Deformation-1995} showed that for a smooth projective surface $S$ with $H^1(S, \mathscr{O}_S)\otimes H^0(S, T_S)=0$ and $H^0(S, \omega_S^\vee)=0$ (for example when $S$ is of general type, or an Enriques surface), the natural map between the Kuranishi spaces $\Def(S)\to \Def(\Hilb^nS)$ is an isomorphism (as germs of analytic spaces).
	\item Hitchin \cite[\S 4.1]{Hitchin-Deformation-2012} showed that for a compact complex surface with $H^1(S, \mathscr{O}_S)=0$, we have a split short exact sequence
	\begin{equation}
	0\to H^1(S, T_S)\to H^1(\Hilb^nS, T_{\Hilb^nS})\to H^0(S, \omega_S^\vee)\to 0.
	\end{equation}
\end{itemize}
Clearly, both results can be recovered from \Cref{thm:main-deformation}.

\begin{proof}[Proof of \Cref{thm:main-deformation}]
	Specializing to $L=\omega^\vee_S$ in \Cref{prop:TwistedHodgeGroupsOfHilb}, we get
	\begin{equation}
	H^q(\Hilb^nS, T_{\Hilb^nS})\cong H^{2n-1, q}(\Hilb^nS, \omega^\vee_n)\cong 
	\bigoplus_{\substack{\lambda\dashv n\\\lambda=(1^{a_1}\dots r^{a_{r}})}} \bigoplus_{\substack{\sum p_k=n-1+|\lambda|\\\sum q_k=q+|\lambda|-n}}\bigotimes_{k=1}^r H^{p_k, q_k}(S^{(a_k)},  \omega^{-k}_{(a_k)}).
	\end{equation}
	Now in the summation,  we can assume $p_k\leq 2a_k$ for any $k$. Therefore, 
	\begin{equation}
	2\sum_k a_k\geq \sum_k p_k=n-1+|\lambda|=-1+\sum_k (k+1)a_k,
	\end{equation}
	hence $\sum_k (k-1)a_k\leq 1$. As a result, $a_2= 0$ or 1, and $a_k=0$ for all $k\geq 3$. In other words, only two partitions can contribute in the above direct sum, namely,  $\lambda=(1^n)$ and $(1^{n-2} 2^1)$. 
	
	For $\lambda=(1^n)$,  we have $r=1$, $p_1=2n-1$, $q_1=q$, hence the contribution is 
	\begin{equation}
	\label{eqn:Contribution-1}
	H^{2n-1, q}(S^{(n)}, \omega^\vee_{S^{(n)}})\cong H^q(S^n, T_{S^n})^{\mathfrak{S}_n}.
	\end{equation}
	For $\lambda=(1^{n-2}2^1)$, we have $r=2$, $p_1=2n-4$, $p_2=2$, $q_1+q_2=q-1$, hence the contribution is
	\begin{align}
	\label{eqn:Contribution-2}
	\begin{split}
	&\bigoplus_{q_1+q_2=q-1} H^{2n-4, q_1}(S^{(n-2)}, \omega_{S^{(n-2)}}^\vee)\otimes H^{2, q_2}(S, \omega_S^{-2})\\
	\cong  &\bigoplus_{q_1+q_2=q-1} H^{ q_1}(S^{(n-2)}, \mathscr{O})\otimes H^{q_2}(S, \omega_S^{\vee}).
	\end{split}
	\end{align}
	Summing the two contributions \eqref{eqn:Contribution-1} and \eqref{eqn:Contribution-2} proves \eqref{eqn:main-Cohomology-Tangent}. Then \eqref{eqn:H^0T}, \eqref{eqn:H^1T} and \eqref{eqn:H^2T} follow immediately.
	%The assertion on the Schouten--Nijenhuis bracket follows from the naturality of the isomorphisms in the proof.
\end{proof}

\subsection{Examples}

\begin{itemize}
	\item For $S=\PP^2$, which is rigid, its Hilbert scheme however has non-trivial deformations, related to the non-commutative deformations of $\PP^2$ via Sklyanin algebras by Nevins--Stafford \cite{NevinsStafford} and Naeghel--Van den Bergh \cite{Naeghel-VdB-2005}. More generally, as explained in Hitchin \cite{Hitchin-Deformation-2012}, Poisson structures on a complex surface $S$, related to its non-commutative deformations, give rise to geometric deformations of $\Hilb^nS$. See \cite{LiChunyi-DeformationHilbDelPezzo} when $S$ is a del Pezzo surface.
	\item If $S$ is a surface of general type, or regular of Kodaira dimension 1, or an Enriques surface, then Fantechi's result applies, and the natural map $\Def(S)\to \Def(\Hilb^nS)$ is an isomorphism. For example, for curves $C_1, C_2$ of genus $\geq 2$, $\Def(\Hilb^n(C_1\times C_2))\cong \Def(C_1)\times \Def(C_2)$, and for a curve $C$ of genus $\geq 3$, $\Def(\Hilb^n(C^{(2)}))\cong \Def(C^{(2)})\cong \Def(C)$. Compare these to the classical fact that $\Def(C)\cong \Def(C^{(n)})$ for a curve $C$ of genus at least 3, by Fantechi \cite{Fantechi-DeformationSymmetricProducts} (generalizing Kempf \cite{Kempf-DeformationSymmetricProducts}).
	\item If $S$ is K\"ahler and has torsion canonical bundle, then $\Hilb^nS$ is also K\"ahler and has torsion canonical bundle. By the Bogomolov--Tian--Todorov theorem, generalized by Ran \cite{Ran} and Kawamata \cite{Kawamata-UnobstructedDeformation}, $\Def(\Hilb^nS)$ is unobstructed (i.e. smooth). 
	\item For $S$ a K3 surface, and $n>1$, $\Def(\Hilb^nS)$ is smooth of dimension 21, hence 1-dimensional higher than $\Def(S)$. In view of Hitchin's result, the anti-canonical section is responsible for the extra direction of deformations. This universal family is most naturally studied in the context of compact hyper-K\"ahler manifolds, see Beauville \cite{Beauville-JDG} and Fujiki \cite{Fujiki-Katata}.
	\item If $S$ is a 2-dimensional complex torus, and $n>1$, $\Def(\Hilb^nS)$ is smooth of dimension 9. Let us describe the deformations for $n>2$: we have $\Hilb^n(S)=K_{n-1}(S)\times^{\Gamma_n} S$, where $\Gamma_n\cong (\ZZ/n\ZZ)^4$ is the group of $n$-torsion points of $S$, acting diagonally on the product of the generalized Kummer variety $K_{n-1}(S)$ and $S$. Note that $K_{n-1}(S)$ is a compact hyper-K\"ahler manifold with $\dim \Def(K_{n-1}(S))=\dim\Def(S)+1=5$. A general deformation of $\Hilb^n(S)$ is of the form $K\times^{\Gamma_n} A$, where $K$ is a deformation of $K_{n-1}(S)$ and $A$ is a deformation of $S$, the diagonal action of $(\ZZ/n\ZZ)^4$ persists after deformations.
	\item If $S$ is a bielliptic surface, since $h^0(S, T_S)=h^1(S, \mathscr{O}_S)=1$ and $H^0(S, \omega^\vee_S)=0$, \eqref{eqn:H^1T} implies that for $n>1$,
	\begin{equation}
	\dim\Def(\Hilb^nS)=\dim \Def(S)+1=\begin{cases}
	3 &\text{ if } \operatorname{ord}(\omega_S)=2;\\
	2 &\text{ if } \operatorname{ord}(\omega_S)=3, 4, 6;
	\end{cases}
	\end{equation}
	The extra (unobstructed) deformation direction is related to the extra deformation direction of the $(2n-1)$-dimensional strict Calabi--Yau manifold constructed in \cite[Theorem 3.5]{OguisoSchroer-EnriquesManifolds}. Note that this case is not covered by the results of Fantechi and Hitchin. See also the discussion in \cite[Section 4.3 and Example 5.6]{BFK}.
\end{itemize}

\section{Nested Hilbert schemes}
\label{sec:NestedHilb}

\subsection{Basic definitions}
Let $S$ be a compact complex surface and $n$ a positive integer. Let $\Hilb^{n, n+1}S$ be the nested Hilbert scheme parametrizing $(\xi, \xi')$ with $\xi\in \Hilb^nS$ and $\xi'\in \Hilb^{n+1}S$ such that the ideal sheaves satisfy $I_{\xi'}\subset I_{\xi}$. $\Hilb^{n, n+1}S$ is sometimes denoted by $S^{[n, n+1]}$ in the literature.  By Cheah \cite{Cheah-CellularDecomp-1998}, $\Hilb^{n, n+1}S$ is a compact complex manifold of dimension $2n+2$.

There are natural morphisms
\begin{align}
\begin{split}
\phi\colon	\Hilb^{n, n+1}S& \to \Hilb^nS\\
(\xi, \xi')&\mapsto \xi,\\
\psi\colon	\Hilb^{n, n+1}S &\to \Hilb^{n+1}S\\
(\xi, \xi')&\mapsto \xi',\\
\rho\colon 	\Hilb^{n, n+1}S &\to S\\
(\xi, \xi')&\mapsto \xi'/\xi,\\
\end{split}
\end{align}
where $\xi'/\xi$ denotes the residual point of $\xi$ in $\xi'$.
It is clear that the morphism
\begin{equation}
\begin{split}
(\phi, \rho)\colon \Hilb^{n, n+1}S &\to \Hilb^nS\times S\\
(\xi, \xi')&\mapsto (\xi, \xi'/\xi)
\end{split}
\end{equation}
is a birational map; in fact, it can be identified with the blow-up morphism of $\Hilb^nS\times S$ along the universal subscheme $Z_n\subset \Hilb^nS\times S$; see Lehn \cite[Proposition 3.8]{Lehn-LectureHilbert}.

Recall that for any line bundle $L$ on $S$, in \eqref{eqn:Ln} we have defined a natural line bundle $L_n$  on $\Hilb^nS$. 
Now given any two line bundles $L, L'$ on $S$, we have the following natural line bundle on $\Hilb^{n, n+1}S$:
\begin{equation}
\phi^*L_n\otimes \rho^*L'.
\end{equation}

\begin{rmk}
	Pulling back line bundles of the form $L_{n+1}$ via $\psi$ does not give extra new line bundles on $\Hilb^{n,n+1}S$. Indeed, from the following commutative diagram
	\begin{equation}
	\xymatrix{
		\Hilb^{n, n+1}S \ar[r]^{(\phi, \rho)} \ar[d]_{\psi} \ar[dr]^{\pi}& \Hilb^nS\times S \ar[d]^{\pi_n\times \id_S}\\
		\Hilb^{n+1}S \ar[dr]_{\pi_{n+1}}& S^{(n)}\times S \ar[d]^{s}\\
		& S^{(n+1)}
	}
	\end{equation}
	together with the fact that $s^*L_{(n+1)}\cong L_{(n)}\boxtimes L$ (as they both pull-back to $L^{\boxtimes (n+1)}$ on $S^{n+1}$)
	we see that there is an isomorphism of line bundles:
	\begin{equation}
	\phi^*L_n\otimes \rho^*L\cong \psi^*L_{n+1}.
	\end{equation}
\end{rmk}

\subsection{Proof of \Cref{thm:main-nested-cohomology}}
The goal of this section is to prove \Cref{thm:main-nested-cohomology}, which determine all the twisted Hodge groups and twisted Hodge numbers of the $\Hilb^{n, n+1}S$ with value in the natural line bundle $\phi^*L_n\otimes \rho^*L'$, for any line bundles $L, L'$ on $S$. We use a similar method as for $\Hilb^nS$ in \Cref{sec:Proof}.

Let $X:=\Hilb^{n, n+1}S$ and $Y:=S^{(n)}\times S$. Consider the composition morphism $\pi:=(\pi_n\times \id_S)\circ (\phi, \rho)$
\begin{equation}
\pi\colon X\to Y.
\end{equation}

As is shown in Cheah \cite{Cheah-VirtualHodgePolynomial}, G\"ottsche \cite{Gottsche-MotiveHilb}, and de Cataldo--Migliorini \cite{deCataldoMigliorini-semismallMotive}, $\pi$ is semismall and admits the following stratification.

For any $\lambda\dashv n$, write $\lambda=(1^{a_1}2^{a_2}\cdots r^{a_r})$ as before, and define
\begin{equation}
I_{\lambda}:=\{j~|~ a_j>0\}\sqcup \{0\}.
\end{equation}
Let $\tilde{P}(n)=\{(\lambda, j)~|~ \lambda\dashv n, j\in I_\lambda\}$.

For any $(\lambda, j)\in \tilde{P}(n)$, set 
\begin{equation}
Y_{\lambda, j}:=\{(z, x)\in S^{(n)}\times S ~|~ \operatorname{mult}_xz=j\}.
\end{equation}
Then 
\begin{equation}
Y=\bigsqcup_{\substack{\lambda\dashv n\\ j\in I_{\lambda}}}Y_{\lambda, j}
\end{equation}
is a stratification by locally closed smooth subvarieties with 
\begin{equation}
\label{eqn:dim-strata-Y}
\dim Y_{\lambda, j}=
\begin{cases}
2|\lambda|+2 \quad &\text{ if } j=0;\\
2|\lambda| \quad &\text{ if } j\neq 0.\\
\end{cases}
\end{equation}

For any $(\lambda, j)\in \tilde{P}(n)$, the restriction of $\pi$ to the preimage of $Y_{\lambda, j}$ gives rise to a fiber bundle
\begin{equation}
\pi_{\lambda, j}\colon X_{\lambda, j}\to Y_{\lambda, j},
\end{equation}
with fibers all isomorphic to 
\begin{equation}
F_{\lambda, j}=
\begin{cases}
\prod_{k=1}^r \mathbb{B}_k^{a_k} \quad &\text{ if } j=0;\\
\prod_{k\neq j} \mathbb{B}_k^{a_k} \times \mathbb{B}_j^{a_j-1}\times \Hilb^{j,j+1}(\CC^2)_0 \quad &\text{ if } j\neq 0,\\
\end{cases}
\end{equation}
where $\mathbb{B}_m$ denotes the Brian\c con variety, and $\Hilb^{j,j+1}(\CC^2)_0$ is the Hilbert scheme parametrizing nested subschemes of $\CC^2$ of length $j$ and $j+1$ supported at the origin, both are irreducible by \cite{Briancon} and by \cite{Cheah-CellularDecomp-1998} respectively, with dimension
\begin{equation}
\label{eqn:dim-Fiber}
\dim F_{\lambda, j}=
\begin{cases}
n-|\lambda| \quad &\text{ if } j=0;\\
n-|\lambda|+1 \quad &\text{ if } j\neq 0.\\
\end{cases}
\end{equation}
Comparing \eqref{eqn:dim-strata-Y} and \eqref{eqn:dim-Fiber}, we see that all strata are relevant. Note also that the variation of Hodge structure is 
\begin{equation}
\label{eqn:VHS}
\dim V_{\lambda, j}\cong
\begin{cases}
\QQ(|\lambda|-n) \quad &\text{ if } j=0;\\
\QQ(|\lambda|-n-1) \quad &\text{ if } j\neq 0.\\
\end{cases}
\end{equation}

\begin{lemma}
	\label{lemma:NormalizationAndPullback-nested}
	Let  $(\lambda, j)\in \tilde{P}(n)$.  Let $L$ and $L'$ be two line bundles on $S$.
	\begin{enumerate}[label=\emph{\roman*})]
		\item If $j=0$, then 
		\begin{align}
		\begin{split}
		\iota_{\lambda, j}\colon S^{(a_1)}\times \cdots\times S^{(a_r)}\times S &\to S^{(n)}\times S\\
		(z_1, \dots, z_r, x) &\mapsto (\sum_{k} kz_k, x)
		\end{split}
		\end{align}
		is a finite birational morphism, and factorizes through the normalization of $Y_{\lambda, j}$.
		We have 
		\begin{equation}
		\iota_{\lambda, j}^*(L_{(n)}\boxtimes L')\cong L_{(a_1)}\boxtimes L^2_{(a_2)}\boxtimes\cdots \boxtimes L^r_{(a_r)}\boxtimes L'.
		\end{equation}
		
		\item If $j\neq 0$, then 
		\begin{align}
		\begin{split}
		\iota_{\lambda, j}\colon S^{(a_1)}\times \cdots S^{(a_j-1)}\times \cdots\times S^{(a_r)}\times S &\to S^{(n)}\times S\\
		(z_1, \dots, z_r, x) &\mapsto (jx+\sum_{k} kz_k, x)
		\end{split}
		\end{align}
		is a finite birational morphism, and factorizes through the normalization of $Y_{\lambda, j}$.
		We have 
		\begin{equation}
		\iota_{\lambda, j}^*(L_{(n)}\boxtimes L')\cong L_{(a_1)}\boxtimes L^2_{(a_2)}\boxtimes\cdots L^j_{(a_j-1)} \boxtimes\cdots\boxtimes L^r_{(a_r)}\boxtimes (L^j\otimes L').
		\end{equation}
	\end{enumerate}
	
\end{lemma}
\begin{proof}
	The assertions about the maps are due to Cheah \cite{Cheah-VirtualHodgePolynomial}. The computation of pullback line bundles is  straightforward and similar to \Cref{lemma:PullBackLineBundle}. We omit the details.
\end{proof}

\begin{proof}[Proof of \Cref{thm:main-nested-cohomology}]
	The overall proof scheme is as the proof of \Cref{thm:main-cohomology}. We only sketch some main steps. 
	
	Step 1. Similarly to \Cref{prop:Isom-HodgeModules}, applying \Cref{thm:BorhoMacPherson-HM} to the semismall map $\pi\colon X\to Y$, and use \Cref{lemma:NormalizationAndPullback-nested} and \Cref{lemma:BirationalFinite}, we get an isomorphism in $\HM^\p(Y, 2n+2)$:
	\begin{align}
	\begin{split}
	\R\pi_*\QQ_X^H[2n+2]\cong &\bigoplus_{\lambda\dashv n} (\iota_{\lambda, 0})_* (\ICH_{S^{(a_1)}}\boxtimes\cdots \boxtimes \ICH_{S^{(a_r)}}\boxtimes \ICH_S)(\QQ(|\lambda|-n))\\
	&\oplus \bigoplus_{\substack{(\lambda,j)\in \tilde{P}(n)\\ j\neq 0}}(\iota_{\lambda, j})_* (\ICH_{S^{(a_1)}}\boxtimes\cdots \boxtimes \ICH_{S^{(a_j-1)}}\boxtimes\cdots\boxtimes\ICH_{S^{(a_r)}}\boxtimes \ICH_S)(\QQ(|\lambda|-n-1))
	\end{split}
	\end{align}
	Similarly to \Cref{cor:IsomDModules}, taking the isomorphism of underlying filtered $\D$-modules:
	\begin{align}
	\label{eqn:IsomD-Module-nested}
	\begin{split}
	\pi_+(\omega_X, F_\bullet)\cong &\bigoplus_{\lambda\dashv n} (\iota_{\lambda, 0})_+ (\ICH_{S^{(a_1)}}\boxtimes\cdots \boxtimes \ICH_{S^{(a_r)}}\boxtimes \ICH_S, F_{\bullet-|\lambda|+n})\\
	&\oplus \bigoplus_{\substack{(\lambda,j)\in \tilde{P}(n)\\ j\neq 0}}(\iota_{\lambda, j})_+ (\ICH_{S^{(a_1)}}\boxtimes\cdots \boxtimes \ICH_{S^{(a_j-1)}}\boxtimes\cdots\boxtimes\ICH_{S^{(a_r)}}\boxtimes \ICH_S, F_{\bullet-|\lambda|+n+1} ).
	\end{split}
	\end{align}
	
	Step 2. Similarly to \Cref{prop:PushforwardOfOmega}, applying the functor $\gr^F_{-p}\circ \DR$ to both sides of \eqref{eqn:IsomD-Module-nested}, and use \Cref{thm:Strictness}, we get an isomorphism in $\operatorname{D^b_{coh}(Y)}$:
	\begin{align}
	\begin{split}
	\R\pi_*\Omega^p_X[2n+2-p]\cong  &\bigoplus_{\substack{\lambda\dashv n\\\lambda=(1^{a_1}\dots r^{a_{r}})}}(\iota_{\lambda, 0})_* \gr^F _{-p-|\lambda|+n} \DR (\ICH_{S^{(a_1)}}\boxtimes\cdots \boxtimes \ICH_{S^{(a_r)}}\boxtimes \ICH_S)\\
	&\oplus \bigoplus_{\substack{(\lambda,j)\in \tilde{P}(n)\\ j\neq 0}}(\iota_{\lambda, j})_* \gr^F_{-p-|\lambda|+n+1} \DR(\ICH_{S^{(a_1)}}\boxtimes\cdots \boxtimes \ICH_{S^{(a_j-1)}}\boxtimes\cdots\boxtimes\ICH_{S^{(a_r)}}\boxtimes \ICH_S)
	\end{split}
	\end{align}
	As a result, we get the analogue of \Cref{prop:PushforwardOfOmega}: in $\operatorname{D^b_{coh}(Y)}$, we have an isomorphism
	\begin{align}
	\label{eqn:RpiOmegap-nested}
	\begin{split}
	\R\pi_*\Omega_{X}^p\cong &\bigoplus_{\substack{\lambda\dashv n\\\lambda=(1^{a_1}\dots r^{a_{r}})}}\bigoplus_{p_0+\cdots p_r=p+|\lambda|-n}(\iota_{\lambda,0})_*\left(\Omega_{S^{(a_1)}}^{[p_1]}\boxtimes\cdots\boxtimes\Omega_{S^{(a_r)}}^{[p_r]}\boxtimes \Omega_S^{[p_0]}[|\lambda|-n]\right)\\
	&\oplus \bigoplus_{\substack{\lambda\dashv n\\\lambda=(1^{a_1}\dots r^{a_{r}})\\j\neq 0, a_j>0}}\bigoplus_{p_0+\cdots p_r=p+|\lambda|-n-1}(\iota_{\lambda,j})_*\left(\Omega_{S^{(a_1)}}^{[p_1]}\boxtimes\cdots\boxtimes\Omega^{[p_j]}_{S^{(a_j-1)}}\boxtimes\cdots\boxtimes\Omega_{S^{(a_r)}}^{[p_r]}\boxtimes \Omega_S^{[p_0]}[|\lambda|-n-1]\right)
	\end{split}
	\end{align}
	
	Step 3. Similarly to \Cref{prop:TwistedHodgeGroupsOfHilb}, we tensor  \eqref{eqn:RpiOmegap-nested} with  line bundle $L_{(n)}\boxtimes L'$ and take hyper-cohomology. Using \Cref{lemma:NormalizationAndPullback-nested} and \Cref{prop:SymmetricPowers}, together with projection formula and K\"unneth formula, we obtain
	
	\begin{align}
	\label{eqn:TwistedHodge-pq-nested}
	\begin{split}
	&H^{p,q}(X,\phi^*L_n\otimes \rho^*L')\\\cong & H^q(X, \Omega_X^p\otimes \pi^*(L_{(n)}\boxtimes L'))\\
	\cong &\bigoplus_{\lambda\dashv n} \bigoplus_{\substack{p_0+\cdots+p_r=p+|\lambda|-n\\q_0+\cdots+q_r=q+|\lambda|-n}}\left(\bigotimes_{k=1}^r H^{p_k,q_k}(S^{(a_k)}, L^k_{(a_k)})\otimes H^{p_0, q_0}(S, L')\right)\\
	\oplus&\bigoplus_{\substack{\lambda\dashv n\\\lambda=(1^{a_1}\dots r^{a_{r}})\\j\neq 0, a_j>0}}\bigoplus_{\substack{\sum_{k=0}^r p_k=p+|\lambda|-n-1\\\sum_{k=0}^rq_k=q+|\lambda|-n-1}}\left(H^{p_1,q_1}(S^{(a_1)}, L_{(a_1)})\otimes\cdots \otimes H^{p_j,q_j}(S^{(a_j-1)}, L^j)\otimes \cdots\otimes H^{p_r,q_r}(S^{(a_r)} , L^r_{(a_r)})\otimes H^{p_0, q_0}(S, L^j\otimes L')\right).\\
	\end{split}
	\end{align}
	
	Step 4.   Multiplying  \eqref{eqn:TwistedHodge-pq-nested} by $x^py^qt^n$, and summing over all $p, q, n\in \mathbb{N}$, one can conclude by an elementary but slightly more tedious computation similar to the end of Proof of \Cref{thm:main-cohomology} in \Cref{sec:Proof}. 
\end{proof}

\section{Final remarks and questions}
In the seminal paper of Ellingsrud, G\"ottsche and Lehn \cite{EllingsrudGottscheLehn}, it is shown that for a compact complex surface $S$, the cobordism class of $\Hilb^nS$ is determined by that of $S$.
As a consequence, for any genus, its value on $\Hilb^nS$ is determined by its value on $S$. For the $\chi_y$-genus, the relation is given by \eqref{eqn:ChiyGenusCoeff} in the introduction. More generally, the case of elliptic genus is worked out by Borisov--Libgober \cite{BorisovLibgober-EllipticGeneraSingular-Duke2003, BorisovLibgober-McKayEllipticGenera-Annals2005}, based on Dijkgraaf--Moore--Verlinde--Verlinde \cite{DMVV}:
\begin{equation}
\sum_{n\geq 0} \operatorname{Ell}(\Hilb^nS)t^n=\frac{1}{\mathbf{L}(\operatorname{Ell}(S), t)}
\end{equation}
where for any power series $f=\sum_{m,l} c_{m,l}q^my^l \in \QQ[[q, y]]$, its Borcherds-type lift is defined as $$\mathbf{L}(f,t):=\prod_{k\geq 1}\prod_{m,l}(1-t^kq^my^l)^{c_{km, l}}.$$

Ellingsrud--G\"ottsche--Lehn \cite{EllingsrudGottscheLehn} actually proved the following stronger statement. Recall that for a vector bundle $F$ on $S$, the \textit{tautological bundle} $F^{[n]}$ on $\Hilb^nS$ is the vector bundle of rank $n\rk(F)$ whose fiber at $\xi\in \Hilb^nS$ is $H^0(\xi, F|_{\xi})$. We have the relation (\cite[Section 5]{EllingsrudGottscheLehn})
\begin{equation}
\label{eqn:TautologicalBundleRelation}
\det(F^{[n]})\cong \det(F)_n\otimes \mathscr{O}(E)^{\otimes \rk(F)}.
\end{equation}
where $\mathscr{O}(E)=\det(\mathscr{O}_S^{[n]})$ is the line bundle associated to the divisor $-\frac{1}{2}D$  with $D$ the exceptional divisor of the Hilbert--Chow morphism. 

\begin{theorem}[{\cite[Theorem 4.1]{EllingsrudGottscheLehn}}]
	\label{thm:EGL-StrongForm}
	Let $S$ be a smooth projective complex surface and let $F_1, \cdots, F_m$ be holomorphic vector bundles on $S$ with $\rk(F_i)=r_i$.  For any polynomial $P$ in Chern classes of $T_{\Hilb^nS}$ and Chern classes of $F_1^{[n]}, \cdots, F_m^{[n]}$, there exists a universal polynomial $\tilde{P}$ depending only on the ranks $r_1, \cdots, r_m$, Chern classes of $T_S$ and the Chern classes of $F_1, \cdots, F_m$, such that 
	\begin{equation}
	\int_{\Hilb^nS}P=\int_S \tilde{P}.
	\end{equation}
\end{theorem}
\begin{rmk}
	Y.-P. Lee and Pandharipande \cite{LeePandharipande-AlgebraicCobordismPairs} developped a more general theory of algebraic cobordism of pairs, where a \textit{pair} consists of a smooth variety and a vector bundle on it. 
	\Cref{thm:EGL-StrongForm} implies that for any vector bundle $F$ on $S$, the cobordism class of the pair $(\Hilb^nS, F^{[n]})$ is determined by the cobordism class of the pair $(S, F)$. In particular, thanks to \eqref{eqn:TautologicalBundleRelation}, for any line bundle $L$, the cobordism class of $(\Hilb^nS, L_n)$ is determined by the cobordism class of $(S, L)$.
\end{rmk}

In \cite[Theorem 1.3]{Gottsche-RefinedVerlinde-2020}, G\"ottsche established a formula computing the \textit{elliptic genus with coefficients} $\operatorname{Ell}(\Hilb^nS, L_n\otimes \mathscr{O}(rE))$, for any $r\in \ZZ$, hence in particular a formula for $\chi_y(\Hilb^nS, L_n\otimes \mathscr{O}(rE))$ (\cite[Corollary 1.4]{Gottsche-RefinedVerlinde-2020}).

\begin{question}
	Can we refine G\"ottsche's formula for $\chi_y(\Hilb^nS, L_n\otimes \mathscr{O}(rE))$ by computing the following twisted Hodge groups
	\begin{equation}
	H^{p,q}(\Hilb^nS, L_n\otimes \mathscr{O}(rE))
	\end{equation}
	or at least their dimensions?
	
	By \eqref{eqn:TautologicalBundleRelation}, it is equivalent to computing for any vector bundle $F$ on $S$, the twisted Hodge groups:
	\begin{equation}
	H^{p,q}(\Hilb^nS, \det(F^{[n]})).
	\end{equation}
	
\end{question}

\begin{question}
	Can we compute the following cohomology groups
	\begin{equation}
	H^q(\Hilb^nS, \Omega^p\otimes F^{[n]}),
	\end{equation}
	in terms of cohomology groups on powers of $S$ with values in some natural coherent sheaves involving $F$?
\end{question}

\bibliographystyle{abbrv}
\bibliography{main}

\medskip \medskip

\noindent{Universit\'e de Strasbourg, Institut de recherche mathématique avancée (IRMA),  France} 

\medskip \noindent{\texttt{lie.fu@math.unistra.fr}}

\end{document}
%%% Local Variables:
%%% mode: latex
%%% TeX-master: t
%%% End: